\documentclass[reqno,10pt]{amsart}
\usepackage{amssymb,amsmath,amsthm,enumerate,amscd,graphicx,color}
\usepackage[usenames,dvipsnames,svgnames,table]{xcolor}
\usepackage{enumitem}
\usepackage{pdfsync}
\usepackage{float}
 \usepackage[colorlinks=true, pdfstartview=FitV, linkcolor=blue, citecolor=blue, urlcolor=blue]{hyperref}
\theoremstyle{rem}
\numberwithin{equation}{section}

\newtheorem{theorem}{Theorem}[section]
\newtheorem{corollary}{Corollary}[section]
\newtheorem{lemma}{Lemma}[section]%
\newtheorem{proposition}{Proposition}[section]
\theoremstyle{definition}
\theoremstyle{remark}

\newtheorem{remark}{Remark}[section]
\newtheorem{example}{Example}[section]

\newcommand{\B}{\mathcal{B} }
\newcommand{\bC}{{\mathbb C}}
\newcommand{\bZ}{{\mathbb Z}}
\newcommand{\bN}{{\mathbb N}}

\newcommand{\End}{\operatorname{End}}

\begin{document}

\author{Natasha Rozhkovskaya}
\address{Department of Mathematics, Kansas State University, Manhattan, KS 66502, USA}
\email{rozhkovs@math.ksu.edu}

\thanks{}

\begin{abstract}
We study the effect of linear transformations on quantum fields with applications to vertex operator presentations 
of symmetric functions. Properties of linearly transformed quantum fields and corresponding transformations of  Hall-Littlewood polynomials are described, including preservation of commutation relations,  stability, explicit  combinatorial formulas and generating functions. We prove that specializations of linearly transformed Hall-Littlewood polynomials
describe all polynomial tau functions of the KP and the BKP hierarchy. Examples of linear transformations are related  to multiparameter symmetric functions, Grothendieck polynomials, deformations by cyclotomic polynomials, and  some other
variations of Schur symmetric functions that exist in the literature. 
\end{abstract}
\title {Linear transformations of vertex  operator  presentations of Hall-Littlewood Polynomials}

\maketitle

\section{Introduction}

 The language of quantum fields is a widely used  tool in representation theory of infinite-dimensional algebraic structures. 
 Actions of algebras of fermions, Heisenberg algebra, Virasoro algebra, $gl_\infty$, affine Lie algebras,  boson-fermion correspondence, vertex algebras and their representations are   some  examples  that use    this language. Presentation of    families of symmetric functions as  results of  application of quantum fields to a vacuum vector allowed researchers to prove many important results with applications in integrable systems and representation theory. A classical example  is  the vertex operator  presentation of   Schur functions by the action of   charged free fermions  \cite{Jing3, Zel}, which is the base of the construction of the boson-fermion correspondence.  From the existing  numerous   vertex operator   presentations  of  other families of  symmetric functions   we mention  through the paper  the ones that are most related  to our construction.

In this note we study  the effect of a   linear transformation of  quantum fields  of vertex operators on  the properties  of the resulting symmetric functions.  Surprisingly, this simple modification covers a broad class of Schur-like symmetric functions that appear in the literature.  At the same time, the simplicity of this transformation allows one to get    many important  properties of these  transformed families of symmetric functions  ``almost for free''  from the properties of the  original  classical  family.

The initial motivation for  this project was  the  study of  tau-functions of the KP and the BKP hierarchy \cite{DJKM4, DJKM1, DJKM3, DJKM2, JM,Sato} by  generalizing the  methods of 
  \cite{NRQ},  where  the author proved that the  multiparameter Schur $Q$-functions  are tau-functions of the BKP hierarchy. We aimed to provide a description of  all polynomial tau-functions of the KP and the  BKP hierarchies  unifying the ideas of \cite{NRQ} and  \cite{KLR}.  This is done  in Section   \ref{Sec-6}. Along the way it was convenient to consider a  more general set up applying   linear transformations to   vertex operators  of Hall-Littlewood polynomials that were first constructed in \cite{Jing2}. This allowed us  not only to formulate and prove in a uniform way  statements  for the KP and  the BKP tau-functions, but to prove a number of properties of linearly transformed  Hall-Littlewood polynomials, such as combinatorial formulas, generating functions, stability property, preservation of commutation relations of  quantum fields under linear transformations. Moreover, these  properties carry on to particular specializations that provide deformations of Hall-Littlewood polynomials, Schur  and   Schur $Q$-functions that     appear in the literature.

  The paper is organized as follows.  In Section \ref{Sec1} we review the necessary facts on symmetric functions, quantum fields,  formal distributions. In Section
 \ref{S-change} we introduce  linear transformations of quantum fields  and describe their basic properties.  In Section \ref{Sec-HL}  we review  Hall-Littlewood polynomials and apply  linear transformations to their vertex operator presentations to define a new family of symmetric functions that depend on a parameter $t$. The properties of these new symmetric functions are discussed. In  Section \ref{Sec-5}   we formulate the properties of  specializations at $t=0$  and $t=-1$.  In  Section \ref{Sec-6} we   describe  all polynomial tau-functions of the KP and the BKP hierarchies as results of   linear transformations of vertex operator presentations of Schur  functions and Schur Q-functions. In  Section \ref{Sec-Ex}  we discuss   particular examples of linear transformations  matching  with the existing  literature, that include  multiparameter symmetric functions, Grothendieck polynomials, deformations by cyclotomic polynomials.
  
 \subsection*{Acknowledgments}
 The author is grateful to the hospitality of the Institut des Hautes \'Etudes Scientifiques  and to  the International Laboratory of Cluster Geometry at the HSE University.  
 
\section{Symmetric functions and quantum fields}\label{Sec1}
\subsection{Symmetric functions}\label{Sec-sym1}
 We review properties  of symmetric functions  following \cite{Md, Stan}. The   setup  is similar to   \cite{JR-genA, KLR, NC, NRQ}.
Consider the  algebra of formal power series $\bC[[{\bf x}]]=\bC[[x_1,x_2,\dots]]$. 
Let $\lambda=(\lambda_1\ge \dots\ge \lambda_l> 0)$ be a partition  of length $l$.
The {\it monomial symmetric function} 
is  a formal series
\[
m_\lambda=\sum_{(i_1,\dots,i_l)\in \mathbb N^l} x^{\lambda_1}_{i_1}\dots x^{\lambda_l}_{i_l}.
\]
Let  $\Lambda$ be the  subalgebra  of $\bC[[{\bf x}]]$  spanned as a vector space  by all  {monomial symmetric functions}. It is  called the  {\it algebra of symmetric functions}.  Note that elements  of $\Lambda$ are invariant with respect to any permutation of  a finite number of  indeterminates 
 $x_1, x_2, \, \dots$. The following  families of symmetric functions play important role in our study. 

For a partition $\lambda=(\lambda_1\ge \dots\ge \lambda_l> 0)$,  {\it Schur symmetric   function} $s_\lambda$ is defined as 
\begin{align}\label{defshur}
s_\lambda(x_1,x_2,\dots ) =\sum_{T} {\bf x}^{T},
\end{align}
where the sum is over all semistandard tableaux of shape $\lambda$.

 {\it Complete symmetric functions}  $h_k= s_{(k)}$ are given by the formula
\[
h_k(x_1,x_2\dots)=\sum_{1\le i_1\le \dots \le i_k<\infty} x_{i_1}\dots x_{i_k},  \quad k\in \bN,
\]
while {\it elementary symmetric functions} $e_k= s_{(1^k)}$  by
 \[e_k(x_1,x_2\dots)=\sum_{1\le i_1< \dots < i_k<\infty} x_{i_1}\dots x_{i_k}, \quad k\in \bN.
\]
 {\it Power sums} $p_k$  are symmetric functions defined by 
\[
p_k(x_1,x_2,\dots)=\sum_{i\in \bN} x_i^k, \quad k\in \bN.
\]
It is  convenient  to set 
$h_{-k}(x_1,x_2\dots)=e_{-k}(x_1,x_2\dots)=p_{-k}(x_1,x_2\dots)=0
$
 for 
$ k\in \bN$  and $h_0=e_0=p_0=1$.

Algebra   $\Lambda$  is  a polynomial algebra in any of these three families of generators: 
\[
 \Lambda=\bC[h_1, h_2,\dots]=\bC[e_1, e_2,\dots]=\bC[p_1,p_2,\dots].
\]
Schur symmetric functions $\{s_\lambda\}$ labeled by all partitions  form a  linear basis of $\Lambda$.
They can be also expressed through complete symmetric functions  by the {\it Jacobi\,-\,Trudi identity}
\begin{align}\label{JT}
s_\lambda=\det[h_{\lambda_i-i+j}]_{1\le i,j\le l}
\end{align}
We will use (\ref{JT}) as the extension of the  definition of $s_\lambda$ for any integer vector $\lambda\in \bZ^l$.

There is a natural  scalar product on   $\Lambda$  where  the set of Schur symmetric functions $\{s_\lambda\}$ labeled by partitions $\lambda$ form an orthonormal basis,
$
 	<s_\lambda,s_\mu >=\delta_{\lambda, \mu}.
$
Then for any  linear operator  acting on the  vector space $\Lambda$ one can define the corresponding adjoint operator. In particular,  any symmetric function  $f\in \Lambda $  defines an operator of multiplication
$
f:g\mapsto fg$  for any $g\in \Lambda$.
The corresponding   adjoint operator $f^\perp$  is defined by the standard rule
$<f^\perp g_1, g_2>=<g_1,fg_2>$ for all   $g_1,g_2\in \Lambda$.

It is known  \cite{Md},  I.5  Example 3, that
$
p_n^\perp= n\frac{\partial}{\partial p_n}.
$
Since any element  $f\in \Lambda$  can be expressed as a polynomial  function of power sums 
\begin{align*}
f= F(p_1,p_2, p_3, \dots),
\end{align*}
the corresponding adjoint operator $f^\perp$ is a polynomial differential operator with constant coefficients
\begin{align*}
f^\perp= F(\partial/\partial p_1,2\partial/\partial p_2, 3\partial/\partial p_3,\dots).
\end{align*}
In particular, $e_k$ and $h_k$ are  homogeneous polynomials of degree $k$ in  $(p_1,p_2, p_3,\dots )$, so
 the adjoint operators $e^\perp_k$ and $h^\perp_k$ are  homogeneous polynomials of degree $k$ in $(\partial/\partial p_1,2\partial/\partial p_2, \dots)$, which implies
  the following statement. 
\begin{lemma}\label{eh_dp}
For any symmetric function $f\in \Lambda$   there exists  a positive  integer   $N= N(f)$, such that
 \[e^\perp_l(f)=0\quad \text{and}\quad h^\perp_l(f)=0\quad \text{ for all\quad   $l\ge N$}.
 \] 
 \end{lemma}
 
 \subsection{Formal distributions and quantum fields } \label{Sec_distributions}
 For more details see \cite{Kac-begin, Kac-bomb}.
Let  $W$ be a vector space. A  {\it $W$-valued formal distribution} is a biletaral series 
in the  indeterminate $u$ with coefficients  in $W$:
\[
a(u)=\sum_{n\in \bZ}a_n u^n,\quad a_n\in W.
\]
We  denote as $W[[u,u^{-1}]]$ the vector space of  all $W$-valued formal distributions.
We also  use the notation  $W[u]$ for the space  of polynomials,  $W[[u]]$ for the space of power series, $W[u, u^{-1}]$ for the space of Laurent polynomials, and $W((u))$   for the space of formal Laurent series.

A special case of a formal distribution  is a {\it quantum field}, which is an $\End\, W$-valued formal distribution  $\Gamma(u)=\sum_{k\in \bZ}\Gamma_k u^{-k}$, such that
 for any $f\in W$,  $\Gamma_{k}(f)=0$ for  $k>>0$.

A formal distribution in two and more  indeterminates is defined similarly. The {\it formal delta-function}  $\delta(u,v)$ is the $\bC$-valued  formal 
distribution  in  variables  $u$ and $v$ 
\begin{align}\label{delta}
\delta(u,v)= \sum_{\substack{i,j\in \bZ\\ i+j=-1 }} {u^i}{v^{j}} = i_{u,v}\left(\frac{1}{u-v}
\right) -i_{v,u} \left(\frac{1}{u-v}
\right),
\end{align}
where 
$i_{u,v}$ (resp. $i_{v,u}$) denotes  the expansion  of a rational function of $u,v$  in the  domain  $|u|>|v|$ (resp. $|u|>|v|$), 
\begin{align}\label{u-v}
i_{u,v}\left(\frac{1}{u-v}\right)=\sum_{k=0}^{\infty}\frac{v^k}{u^{k+1}}.
\end{align}

\subsection{Generating series  of polynomial  differential operators acting  on $\Lambda$} \label{secGKP}

Denote by $\mathcal D$  the algebra  of differential operators  acting on $\Lambda=\bC[p_1,p_2,\dots]$, which consists of finite sums  
\[
\sum_{i_1,\dots, i_m} F_{i_1, \dots i_m}(p_1, p_2,\dots) {\partial _{p_1}^{i_1}}\dots  {\partial _{p_m}^{i_m}},
\] 
where coefficients  $F_{i_1, \dots i_m}(p_1, p_2,\dots)$ are  polynomials in $(p_1, p_2,\dots)$.
Then operators of multiplication $p_n, h_n, e_n$, their adjoints $p_n^\perp, h_n^\perp, e_n^\perp$ along with  their products  are elements of   $\mathcal D$.

Consider the generating series  of  complete and  elementary symmetric functions
\begin{align}\label{HE}
H(u)=\sum_{k\in \bZ_{\ge 0}}  {h_k}{u^k}=\prod_{i\in \bN} \frac{1}{1-x_iu},\quad \quad 
E(u)=\sum_{k\in \bZ_{\ge 0}}{ e_k}{ u^k}=\prod_{i\in \bN} {(1+x_iu)},
\end{align}
which are elements of  $ \Lambda[[u]]$. We will use the same notation for  the generating series of the corresponding multiplication operators
$H(u),E(u)\in \mathcal D[[u]]$. Similarly, we define $E^\perp(u), H^\perp(u)\in \mathcal D[[u^{-1}]]$ as
\begin{align}\label{perpHE}
  E^\perp(u)= \sum_{k\in \bZ_{\ge 0}}\frac {e^\perp_k} {u^k},\quad H^\perp(u)= \sum_{k\in \bZ_{\ge 0}} \frac{h^\perp_k} {u^k}.
\end{align}

The following properties of these   generating series with  coefficients in $\mathcal D$ are well known  (e.g. \cite {Md}, I.5). 
\begin{proposition}\label{prop_rel} We have in  $\mathcal D[[u]]$ (resp. in  $\mathcal D[[ u^{-1}]]$\,) 
\begin{align*}
H(u) E(-u)=1,
\quad 
H^\perp(u) E^{\perp}(-u)=1,
\end{align*}
\begin{align*}
H(u)&= exp\left(\sum_{n\in \bN} \frac{p_n}{n}{u^n}\right),\quad 
E(u)= exp\left(-\sum_{n\in \bN} \frac{(-1)^{n} p_n}{n}{u^n}\right),
\end{align*}
\begin{align}\label{DHEP}
  E^\perp(u)= exp\left(-\sum_{k\in \bN} {(-1)^k} \frac{\partial}{\partial p_k} \frac {1}{u^k}\right),
\quad     H^\perp(u)= exp\left( \sum_{k\in \bN}\frac{\partial}{\partial p_k} \frac{1}{u^k}\right).
\end{align}
\end{proposition}
\begin{lemma} \label{propHE}  (\cite {Md}, I.5 Example 29). We have  the following commutation relations in   $\mathcal D[[u^{-1}, v]]$:

  \begin{align*}
\left(
1-\frac{v}{u}
\right)E^\perp(u)E(v)= E(v)E^\perp(u),\\
\left(
1-\frac{v}{u}
\right)H^\perp(u)H(v)= H(v)H^\perp(u),\\
H^\perp(u)E(v)= \left(
1+\frac{v}{u}
\right)E(v)H^\perp(u),\\
E^\perp(u)H(v)= \left(
1+\frac{v}{u}
\right)H(v)E^\perp(u).
\end{align*}
\end{lemma}

\subsection{Schur symmetric $Q$-functions }\label{Sec-sQ}
The elements $\{q_k(x_1,x_2,\dots)\}_{k\in \bZ}$ are the  coefficients of the expansion of $Q(u)\in  \Lambda[[u]]$, where
 \begin{align}\label{shurq}
Q(u) =\sum_{k\in \bZ} q_k u^k= E(u) H(u).
\end{align}
Note that $q_k=\sum_{i=0}^k e_ih_{k-i}$ for $k>0$, $q_0=1$, and $q_k=0$ for $q<0$.
For $a,b\in \bZ_{\ge 0}$  let
\begin{align}\label{qrs}
q_{a,b}=q_aq_b+2\sum_{i\in \bZ} (-1)^i q_{a+i} q_{b-i}.
\end{align}
Then,   \cite{Md},  III.8,
\begin{align}\label{qab_qba}
q_{a,b}=-
q_{b,a}, \quad q_{a,a}=0.
\end{align}

\begin{proposition}\label{prop_rel2}  (\cite{Md},  III.8) We have in  $\mathcal D[[u]]$ (resp. in  $\mathcal D[[ u^{-1}]]$\,) 
  \begin{align*}
Q(u) =S_{odd}(u)^2, \quad\text{where}\quad   
S_{odd}(u)=exp\left(\sum_{n\in \bN_{odd}}\frac{p_{n}}{n}{u^{n}}\right),
\end{align*}
\begin{align*}
S_{odd}^{\perp}(u)=exp\left(\sum_{n\in \bN_{odd}} \frac{\partial}{\partial p_n}\frac{1}{u^{n}}\right).
\end{align*}
Here $\bN_{odd}=\{1,3,5,\dots\}$.
\end{proposition}

Recall that the {\it  Pfaffian}  of a skew-symmetric  matrix $M=[M_{ij}]$ of size $2l\times 2l$ is defined as
$$
\mathrm{Pf}\,[M]=\sum_{\sigma\in S'_{2l}}  sgn(\sigma) M_{\sigma(1) \sigma(2)}  \cdots M_{\sigma(2l-1) \sigma(2l)},
$$
where $S'_{2l}$  is the subset of the permutation group $S_{2l}$ that consists of $\sigma\in S_{2l}$  such that $\sigma(2k-1)<\sigma(2k)$ for $1\le k\le l$ and
$\sigma(2k-1)<\sigma(2k+1)$ for $1\le k\le l-1$. 

If $\lambda=(\lambda_1,\dots, \lambda_{2m})$ is a strict partition, i.e. $\lambda_1>\dots> \lambda_{2m}\ge 0$, then the matrix 
$M_\lambda= (q_{\lambda_i,\lambda_j})$ is skew-symmetric by (\ref{qab_qba}), and the {\it Schur symmetric  $Q$-function} $q_\lambda$  is defined as
\begin{align}\label{defq}
q_\lambda(x_1, x_2, \dots)=\text{Pf} \,M_\lambda.
\end{align}


\subsection{Charged free fermions}\label{Sec_ferminons}

 Let  $z$ and $u$ be   formal indeterminats. 
Consider the  {\it boson Fock space} $\B= \bC[ z, z^{-1}]\otimes \Lambda$, where $\Lambda$ is the ring of symmetric functions.
A number of important algebraic structures act on the space $\B$. We review the action of charged free fermions  and refer to   \cite{ JM, Kac-begin, Kac-bomb}
for more details. 

Let $R(u)$ and  its inverse  be  formal distributions in variable $u$  of operators  acting on the  elements of the  form  $z^m f$,  where $f \in \Lambda$,  $m\in \bZ$, by the rule
\[
R(u) (z^mf)={ z^{m+1}}{u}^{m+1} f,
\quad 
R^{-1}(u) (z^mf)={ z}^{m-1} u^{-m}f.
\]
Define  formal distributions $\psi^\pm(u)$  of operators  acting on the  space $\B$  through the action of  $R^{\pm 1}(u)$ and the $\mathcal D$-valued generating series (\ref{HE}), (\ref{perpHE}):
\begin{align*}
\psi^+(u)&=u^{-1}R(u)H(u) E^{\perp}(-u),\\
\psi^-(u) &=R^{-1}(u)E(-u) H^{\perp}(u),
\end{align*}
or, in  other words, for any $m\in \bZ$ and any $f\in \Lambda$,
\begin{align*}
\psi^+(u)(z^mf)&=z^{m+1}u^{m}H(u) E^{\perp}(-u)(f),\\
\psi^-(u)(z^mf)&=z^{m-1}{u^{-m}}E(-u) H^{\perp}(u)(f).
\end{align*}
Let the operators   $\{\psi^{\pm}_{i}\}_{i\in \bZ+1/2}$ be the coefficients of the expansions
 \[
\psi ^\pm(u)= \sum_{i\in \bZ+1/2}\psi^{\pm}_i u^{-i-1/2}.
\] 
These operators are called the {\it  charged free fermions}. Formal distributions $\psi^\pm(u)$ of  operators acting on  the space  $\B$ are  quantum fields that 
satisfy relations 
\begin{align*}
\psi^\pm(u)\psi^\pm(v)+ \psi^\pm(v)\psi^\pm(u)&=0,
\\
\psi^+(u)\psi^-(v)+ \psi^-(v)\psi^+(u)&=\delta(u,v),
\end{align*}
or, equivalently,   
\begin{align}\label{ckl}
\psi_k^\pm\psi_l^\pm +\psi_l^\pm\psi_k^\pm=0,\quad 
\psi_k^+\psi_l^- +\psi_l^-\psi_k^+=\delta_{k, -l}, \quad k,l\in \bZ+1/2.
\end{align}

\subsection{Heisenberg algebra}  \label{Sec_Hies}

The Heisenberg algebra  is the complex Lie algebra with  a basis $\{\alpha_k\}_{ k\in \bZ}\cup\{ 1\}$ and commutation relations
\begin{align}
[1,\alpha_n]=0, \quad [\alpha_m, \alpha_n]=m\delta_{m,-n} =1, \quad m,n\in \bZ. \label{hei_rel}
\end{align}
This is equivalent to 
\[
[\alpha(z), \alpha(w)]=\partial_w \delta(z,w)\cdot 1, 
\]
where $\alpha(z)=\sum_{n\in\bZ}\alpha_n z^{-n-1}$.
The Heisenberg algebra acts on the space $\B$
by differentiation and multiplication operators
\begin{align*}
\alpha_k=\partial/\partial p_k, \quad \alpha_{-k}= kp_k, (k=1,2,\dots),\quad  \alpha_0=a_0\cdot 1,\quad 1=1.
\end{align*}

\subsection{Neutral fermions} \label{Sec_neutral}
Consider the  boson Fock space  generated by odd power sums:
  \[\B_{odd}= \bC[  p_1, p_3, p_5,\dots].
   \]
    Recall   \cite{Md} III.8 (8.3) that   $q_k\in \B_{odd}$,  and  that $\B_{odd}= \bC[q_1,q_3,\dots ]$.
  From (\ref{DHEP}) it is clear that  $\B_{odd}$  is invariant with respect  to action of  $e^\perp_k$,  and $ h^\perp_k$, and 
one can   prove 
that restrictions to  ${\B_{odd}}$ of the operators
$
 E^{\perp}(u)$,  $H^{\perp}(u)$, $S_{odd}^{\perp}(u)$
coincide.

Define a quantum field $\varphi(u)$ of operators acting on $\B_{odd}$:
 \begin{align}\label{defphi}
&\varphi(u)=  E(u)H(u) E^{\perp}(-u)=Q(u) S_{odd}^{\perp}(-u).
\end{align}
Let  $\{\varphi_i\}_{i\in \bZ}$ be
coefficients  of the expansion $
\varphi(u)=\sum_{j\in \bZ}\varphi_j u^{-j}$.

One has relations
\[
 \varphi(u)\varphi(v) + \varphi(v)\varphi(u) =2v\delta(v,-u),
 \]
where 
$\delta(u,v)$ is  the formal delta function. 
Hence (\ref{defphi}) is the action of  the Clifford algebra of neutral fermions on the space $\B_{odd}$:
 \begin{align}\label{neut1}
\varphi_m \varphi_n+\varphi_n \varphi_m= 2(-1)^m \delta_{m+n,0}\quad \text {for}\quad m, n\in \bZ.
\end{align}


\section{Linear transformations of quantum fields}\label{S-change}
\subsection{Linear transformations of quantum fields}
Let  $\Gamma(u)=\sum_{i\in \bZ} \Gamma_i u^{-i}$ be   a quantum field of operators $\{\Gamma_i\}_{i\in\bZ}$ acting on a  vector space $W$. Fix $f\in W$ and set
\begin{align}\label{Fu}
\Gamma (u_1)\dots \Gamma (u_l) \,  (f)=  \mathcal F(u_1,\dots, u_l) .
\end{align}
 Note that  (\ref{Fu})  is a well-defined  formal distribution with coefficients in $W$:
 \begin{align}\label{Fexpand}
\mathcal F(u_1,\dots, u_l)=\sum_{\lambda\in \bZ^l} F_\lambda\, u_1^{\lambda_{1}} \dots u_{l}^{\lambda_{l}},\quad F_\lambda=\Gamma_{-\lambda_1}\dots \Gamma_{-\lambda_l}\,(f) \in W.
 \end{align}

Let $A=(A_{ij})_{i,j\in \bZ}$ be an  infinite  complex-valued  matrix. 
Set formally
\begin{align*}
\tilde \Gamma_{i}=\sum_{i\in \bZ}A_{i,j} \Gamma_{j}. \quad 
\end{align*}

For any integer vector $\lambda \in \bZ^l$ consider a formal infinite sum
 \begin{align*}\label{defn1}
 \tilde F_{\lambda}&=\sum_{\alpha \in \bZ^l} A_{-\lambda_1,-\alpha_1}\dots A_{-\lambda_l,
 -\alpha_l}   F_{\alpha},
  \end{align*}
  where $\{F_{\alpha}\}\subset W$ are coefficients of the expansion  (\ref{Fexpand}).
  
  We say that $B=(B_{ij})_{i,j\in\bZ}$ is the (left) {\it  inverse }of the infinite matrix $A$  and write $B=A^{-1}$
if 
$
\sum_{k\in\bZ} B_{ik} A_{kj}=\delta_{i,j}.
$
 \begin{theorem}\label{matrixA}
 
   Assume that  for any fixed $i\in \bZ$, $A_{i,j}=0$ for $j<<0$:
\[
\begin{array}{r|ccccccccc}
\downarrow i & j \rightarrow &\dots &-2 &-1 &0 &1&2 &3&\dots \\
  \hline
\vdots&    \dots &\dots &\dots &\dots &\dots &\dots &\dots &\dots &\dots \\
\vdots&  \dots &0& * & *  & *& *& * & * &\dots \\
-3&   \dots &0 &0 & 0  & *  & *  &* & *  &\dots \\
-2&   \dots &0 &0&* & *  & * & *  & *&\dots \\
-1&   \dots &0&0&* & * &* &*& * & \dots\\
0&  \dots &0 &0&0  &*&*&*&* &\dots \\
1&  \dots &0 &0  &0  &0&*&*&*&\dots\\
2&   \dots &0 & 0   &0&0&*&*&*&\dots \\
3&   \dots &0 & 0   &0 &*&*& *&*&\dots \\
\vdots&    \dots &\dots&\dots &\dots &\dots  &\dots &\dots &\dots&\dots 
\end{array}
\]

   For each $i\in \bZ$, let  $A_{i,M(i)}$ be  the first non-zero term in the  $i$-th row of the matrix $A$, reading from left to right:
 \begin{align}\label{Ms}
M(i)=\max \{k\in \bZ |   A_{i,r}= 0\quad \text{for all}\quad   r<k\}. 
\end{align}    
  Then 
\begin{enumerate}[label=(\alph*)]
\item 
$\tilde \Gamma_i$ is a well-defined linear operator acting on the space $W$ for any $i\in \bZ$.

\item 
  $ \tilde F_{\lambda}$   is a well-defined  finite linear combination of $F_\alpha$'s, i.e.  $ \tilde F_{\lambda}\in W$, 
and 
   \begin{align*}\tilde F_{\lambda}= \tilde\Gamma_{-\lambda_1}\dots \tilde\Gamma_{-\lambda_l}(f). 
    \end{align*}
\item   If $A$ is invertible, then 
   \begin{align*}
F_{\lambda}&=\sum_{\alpha \in \bZ^l} (A^{-1})_{-\lambda_1, -\alpha_1}\dots (A^{-1})_{-\lambda_l, -\alpha_l} \tilde F_{\alpha}.
 \end{align*}
 \item \label{reexp1}
If $A$ is invertible,  then formal distribution  $F(u_1,\dots, u_l)$    can be re-expanded:
   \begin{align*}
\mathcal F(u_1,\dots, u_l)= \sum_{\lambda\in \bZ^l}\tilde F_\lambda \, g_{\lambda_1}(u_1)\dots g_{\lambda_l}(u_l),
 \end{align*}
where 
$
g_k(u)= \sum_{s\in \bZ} (A^{-1})_{-s,-k} u^{s}
$
 are formal  complex-valued distributions.

\item 

 If $M(i)$ is a strictly increasing function of $i$, then $\tilde \Gamma(u)=\sum_{i\in \bZ} \tilde\Gamma_i u^{-i}$  is a quantum field. 
 In that case $\tilde F_{(i,\lambda)}=0$ for any $\lambda\in \bZ^l$ and $i<<0$.
 \[
\begin{array}{r|ccccccccc}
&  &&&&&& & &\\
  \hline
&    \dots &\dots &\dots &\dots &\dots &\dots &\dots &\dots &\dots \\
&  \dots &*& *&* & *& *& * & * &\dots \\
&  \dots &*& * &* & *& * & *  &* &\dots \\
&   \dots &0&* & * & *  & * &* & *  &\dots \\
&   \dots &0&0&0& *  & * & *  & * &\dots \\
&   \dots &0&0&0 & 0 &0 &*& * & \dots\\
&  \dots &0 &0&0 &0&0&0&* &\dots \\
&  \dots &0  &0  &0  &0&0&0&0&\dots\\
&    \dots &\dots&\dots&\dots&\dots  &\dots &\dots &\dots&
\end{array}
\]

\end{enumerate}
\end{theorem}

\begin{proof}
 \begin{enumerate}[label=(\alph*)]
 \item  Since 
 $\Gamma(u)$ is a field,  for any $f\in W$ there exists an integer $N(f)$ such that $\Gamma_{j}(f)=0$ for all $j>N(f)$. 
For any fixed $i\in \bZ$    $A_{i,j}=0$ for $j<M(i)$. Then 
\begin{align}\label{eq_p}
\tilde \Gamma_{i}(f)=\sum_{j\in \bZ}A_{i,j} \Gamma_{j} (f)= \sum_{M(i)\le j\le N(f)}A_{i,j} \Gamma_{j} (f)
\end{align}
 is  a well-defined finite sum  of  elements in  $ W$.

\item  This follows from (a).

\item 
One has
\begin{align*}
&\sum_{\alpha\in \bZ^l} (A^{-1})_{-\lambda_1, -\alpha_1}\dots (A^{-1})_{-\lambda_l, -\alpha_l} \tilde F_{\alpha}\\
&=
\sum_{\alpha, \beta \in \bZ^l} (A^{-1})_{-\lambda_1,-\alpha_1}\dots (A^{-1})_{-\lambda_l,- \alpha_l} A_{-\alpha_1, -\beta_1}\dots A_{-\alpha_l, -\beta_l}\  F_{\beta}=
\sum_{\beta\in\bZ^l} \delta_{\lambda,\beta} F_\beta= F_\lambda.
\end{align*}
\item From (c),
\begin{align*}
\mathcal F(u_1,\dots, u_l)&
=
\sum_{\lambda\in \bZ^l} \sum_{\alpha \in \bZ^l}  (A^{-1})_{-\lambda_1,- \alpha_1}\dots (A^{-1})_{-\lambda_l,-\alpha_l} \tilde F_{\alpha}  u_1^{\lambda_{1}} \dots u_{l}^{\lambda_{l}}\\
&=
\sum_{\alpha\in \bZ^l}\tilde F_\alpha g_{\alpha_1}(u_1)\dots g_{\alpha_l}(u_l).
 \end{align*}

\item 
If $M(i)$ is a strictly increasing function of $i$, then for $i>>0$,   $M(i)>N(f)$  and  all the terms in  (\ref{eq_p}) vanish. Then  
$\tilde \Gamma_{i}(f)=0$ for $i>>0$, so  $\tilde \Gamma(u)$ is a quantum field.
For the second statement observe that $\tilde F_{(i,\lambda)}= \tilde \Gamma_{-i} (\tilde F_\lambda)$.
 \end{enumerate}
\end{proof}

\subsection{Re-expansion of formal delta-function} \label{S-fg}
For  a complex-valued  infinite matrix
 $A=(A_{ij})_{i,j\in \bZ}$ consider the  collection of  formal complex-valued  distributions
\begin{align*}
f_k(x)= \sum_{s\in \bZ} A_{-k,-s} x^s,\quad k\in \bZ.
\end{align*}
If $A$ is invertible, we also introduce  formal complex-valued  distributions
\begin{align*}
g_k(x)= \sum_{s\in\bZ} A^{-1}_{-s,-k} x^{s}, \quad k\in \bZ.
\end{align*}
\begin{lemma}\label{L-delta}
\begin{enumerate}[label=(\alph*)]
\item Assume that $A$ is invertible. Then we have the equality of formal  complex-valued  distributions
\[
\sum_{k\in \bZ} g_k(x^{-1})\,  f_k(y)= x\delta(x,y). 
\]
where $\delta(x,y)$ is  the formal delta function (\ref{delta}).
\item 
If $A$ and $A^{-1}$ are block matrices of the form 

\begin{align}\label{block}
A=
\left[
\begin{array}{c|c|c}
\raisebox{-15pt}{{\Large\mbox{{$A^-$}}}} &\vdots &\raisebox{-15pt}{{\Large\mbox{{$0$}}}} \\
&0&\\
\hline
\dots 0&  1 & 0\dots \\
 \hline
 \raisebox{-15pt}{{\Large\mbox{{$0$}}}} &0 &\raisebox{-15pt}{{\Large\mbox{{$A^+$}}}} \\
 &\vdots&
 \end{array}
\right],
\quad 
\quad A^{-1}=
\left[
\begin{array}{c|c|c}
\raisebox{-15pt}{{\Large\mbox{{$(A^-)^{-1}$}}}} &\vdots &\raisebox{-15pt}{{\Large\mbox{{$0$}}}} \\
&0&\\
\hline
\dots 0&  1 & 0\dots \\
 \hline
 \raisebox{-15pt}{{\Large\mbox{{$0$}}}} &0 &\raisebox{-15pt}{{\Large\mbox{{$(A^+)^{-1}$}}}} \\
 &\vdots&
 \end{array}
\right],
\end{align}
 with $A^-= (A_{i,j})_{i,j<0}$, $A^+= (A_{i,j})_{i,j\ge0}$,  and $A_{0,0}=1$,
then 
\begin{align*}
\sum_{k\ge 0}g_k(x^{-1})\,  f_k(y)=\sum_{k\ge 0}\frac{y^{k}}{x^{k}}=i_{x,y}\left(\frac{x}{x-y}\right),\\ \quad
\sum_{k<0 } g_k(x^{-1})\,  f_k(y)=\sum_{k\ge 0}\frac{x^k}{y^{k}}=i_{y,x}\left(\frac{x}{y-x}\right), \quad
\end{align*}
where $i_{x,y}\left(\frac{1}{x-y}\right)$ is the expansion of a rational function (\ref{u-v}).
\end{enumerate}
\end{lemma}
\begin{proof}
 We check  the first statement:
  \[
\sum_{k\in \bZ} g_k(x^{-1}) f_k(y)= \sum_{k,r,p\in \bZ} (A^{-1})_{-r, -k}  A_{-k,-p} x^{-r}y^{p}= \sum_{r,p\in \bZ} \delta_{p,r} x^{-r}y^{p} 
=x \delta(x,y).
\]
Other identities follow from similar computations.
\end{proof}
 Examples of identities of type  Lemma \ref{L-delta} can be found  in Section \ref{Sec-Ex}.
 
 \begin{remark} If $A$ is a two-block matrix as in Lemma  \ref{L-delta} (b), then 
  $f_{k}(x)$ is  a power sum  in $x$ for $k>0$,   and    in $1/x$ for $k<0$. If, in addition  for any fixed $i\in \bZ$, $A_{i,j}=0$ for $j<<0$, then 
 for $k> 0$  $f_{k}(x)$ is  a  polynomial.
\end{remark}
 
We define matrix $ A^\vee=  (A^\vee)_{ij\in\bZ} $ by
  \begin{align*}
  (A^\vee)_{ij}= A_{-j, -i}.
  \end{align*}
Next statement is obvious.
\begin{lemma}\label{Lem-sym}
Assume that $A$ is invertible and $A^{-1}= A^\vee$. Then
$g_k(x)=  f_{-k} (x^{-1}).
$
\end{lemma}

\subsection{ Commutation relations of  transformed vertex operators}\label{Sec-com}
 In some cases  we can formulate conditions when  the linearly  transformed  quantum fields keep  the commutation relations of the  original quantum fields. In this section we state these conditions for commonly used  algebraic structures of charged free fermions, neutral fermions and the Heisenberg algebra. 

\begin{proposition}
\label{tildeferm}
Let $\{\psi^{\pm}_{r-1/2}\}_{r\in \bZ}$ be  charged free fermions  satisfying relations (\ref{ckl}).  
 Let  $A=(A_{ij})_{i,j\in \bZ}$ and $B=(B_{ij})_{i,j\in \bZ}$   be two matrices, such that  for any fixed $i$, $A_{i,j}=0$ and $B_{i,j}=0$ for $j<<0$.
Set 
\[
\tilde \psi^{+}_{k-1/2}=\sum_{i\in \bZ}A_{k,i}\psi^+_{i-1/2},
\quad 
\tilde \psi^{-}_{k-1/2}=\sum_{i\in \bZ}B_{k,i}\psi^-_{i-1/2}.
\quad 
\]
Then
 $\{\tilde \psi^{\pm}_{r-1/2}\}_{r\in \bZ}$ satisfy the anti-commutation relations of charged free fermions if and only if
 $(A^{-1})_{i,j}= B_{1-j, 1-i}$. 
\end{proposition}
\begin{proof}
Relations  $[\tilde \psi^\pm_{k-1/2},\tilde \psi^\pm_{m-1/2} ]_+=0$, for  $k,m\in \bZ$,  are immediate. One has
 \begin{align*}
 [\tilde \psi^+_{k-1/2},\tilde \psi^-_{m-1/2} ]_+&=\sum_{i,j\in\bZ} A_{k,i}B_{m,j} [ \psi^+_{i-1/2}, \psi^-_{j-1/2} ]_+
=
 \sum_{i,j\in\bZ} A_{k,i}B_{m,j}\delta_{i+j,1}= \sum_{i\in \bZ} A_{k,i}B_{m,1-i}.
 \end{align*}
 Then $ [\tilde \psi^+_{k-1/2},\tilde \psi^-_{m-1/2} ]_+=\delta_{k+m,1}$ if and only if
$
\sum_{i\in \bZ} A_{k,i} B_{1-m, 1-i}=\delta _{k,m}.
$
\end{proof}
\begin{proposition} 
Let $\{\varphi_k\}$ be neutral fermions,   satisfying relations  (\ref{neut1}).
Let  $A=(A_{ij})_{i,j\in \bZ}$  with  the  property that for any fixed $i$, $A_{i,j}=0$  for $j<<0$. Set
 \begin{align*}
\tilde \varphi _{k}=\sum_{i\in \bZ}A_{k,i}\varphi_{i}.
 \end{align*}
Then 
$\{\tilde \varphi_{i}\}_{i\in \bZ}$ satisfy the anticommutation relations of neutral  fermions if and only if
 $(A^{-1})_{i,j}= (-1)^{i-j}A_{-j,-i}$.
 \end{proposition}
\begin{proof}
\begin{align*}
[\tilde \varphi_{k},\tilde \varphi_{m} ]_+=\sum_{i,j\in \bZ} A_{k,i}A_{m,j} [ \varphi_{i}, \varphi_{j} ]_+
=
 2\sum_{i,j\in\bZ} A_{k,i}A_{m,j} (-1)^j\delta_{i+j,0}=  2\sum_{i\in \bZ} A_{k,i}(-1)^i A_{m,- i}.
 \end{align*}
 Then  $2 (-1)^m\delta_{m+k,0}=[\tilde \varphi_{k},\tilde \varphi_{m} ]_+ $
 is equivalent to  
$
  \sum_{i\in \bZ} A_{ki}(-1)^{m-i} A_{-m, -i}=\delta_{k,m}.
  $
\end{proof}
 \begin{proposition} 
Let $\{\alpha_k\}$ be generators of  Heisenberg algebra,   satisfying relations  (\ref{hei_rel}).
Let  $A=(A_{ij})_{i,j\in \bZ}$  with  the  property that for any fixed $i$, $A_{i,j}=0$  for $j<<0$. Set
\[
\tilde \alpha _{k}=\sum_{i\in \bZ}A_{k,i}\alpha_{i}.
\] 
Then 
$\{\tilde \alpha_{i}\}_{i\in \bZ}$ satisfy the  relations of type  (\ref{hei_rel})  if and only if
 $APA^T= P$,  where  $P=(P_{i,j})_{i,j\in\bZ}$ is a matrix with  entries $P_{i,j}= i\delta_{i,-j}$, and $A^T_{ij}= A_{j,i}$. 
 \end{proposition}
 \begin{proof} The statement follows from this calculation: 
 \begin{align*}
 k\delta_{k,-m}&=[\tilde \alpha_{k},\tilde \alpha_{m} ]=\sum_{i,j\in\bZ} A_{k,i}A_{m,j} [ \alpha_{i}, \alpha_{j} ]=
 \sum_{i,j\in\bZ} A_{k,i}i\delta_{i,-j} A_{m,j} =  \sum_{i,j\in \bZ} A_{k,i} i\delta_{i,-j} (A^T)_{j, m}.
 \end{align*}
 \end{proof}

\section{Linear transformations of vertex  operator presentation of Hall-Littlewood polynomials}\label{Sec-HL}
In this section we consider vertex operator presentation of Hall-Littlewood polynomials, constructed first in \cite{Jing2}.
We apply   linear transformations  of Section \ref{S-change}  to vertex operators of Hall-Littlewood polynomials to  obtain new symmetric  polynomials depending on parameter $t$ and deduce their properties. In the subsequent sections we match specializations with different families that are studied by other authors.

  \subsection{Hall-Littlewood polynomials}\label{secHL}
First, we review necessary  facts about Hall-Littlewood polynomials in a set up similar to \cite{Md, NC}.
Let $E(u)$, $H(u)$,  $E^{\perp}(u)$, $H^{\perp}(u)$  be  quantum  fields of operators acting on the space  of symmetric functions $\Lambda$ defined in Section \ref{secGKP}.
Define  quantum fields $ \Gamma^\pm(u) =\sum_{k\in \bZ} \Gamma^\pm_{k}u^{-k} $  of operators  acting on  $\Lambda[[t]]$
\begin{align}
\Gamma^+(u)&= E(-tu)H(u) E^{\perp}(-u),\label{phps1}
\\
\Gamma^-(u)&= H(tu)E(-u) H^{\perp}(u). \label{phps2}
\end{align}

Consider a formal distribution with coefficients in $\Lambda[[t]]$
\begin{align}\label{genHL}
\mathcal{F}(u_1,\dots, u_l; t)&=\prod_{1\le i<j\le l}i_{u_i, tu_j}\left(\frac{u_i-u_j}{u_i-tu_j}\right)\prod_{i=1}^{l} E(-tu_i)H(u_i),
\end{align}
with the series expansion of  rational functions $\frac{u_i-u_j}{u_i-tu_j}$ in the regions $|tu_j|<|u_i|$ for  $1 \le i<j\le l$:
\begin{align*}
i_{u_i, tu_j}\left(\frac{u_i-u_j}{u_i-tu_j}\right)=1+\sum_{s\ge 1}(t^s- t^{s-1})\left(\frac{u_j}{u_i}\right)^s.
\end{align*}

\begin{theorem}\label{ferm_twisted} 
\begin{enumerate}[label=(\alph*)]
\item\label{genHL 3}
 Quantum fields $\Gamma^\pm(u)$ satisfy  generalized fermion  relations 
\begin{align*}
\left(u-{vt}\right)\Gamma^\pm(u)\Gamma^\pm(v)+ \left(v-ut \right)\Gamma^\pm(v)\Gamma^\pm(u)&=0,
\\
\left(v-ut\right)\Gamma^+(u)\Gamma^-(v)+\left(u-vt\right)\Gamma^-(v)\Gamma^+(u)&=\delta(u,v)(1-t)^2.
\end{align*}
\item 
\begin{align}\label{VpHL}
\Gamma^+(u_1)\dots \Gamma ^+(u_l) \,  (1)=  \mathcal{F}(u_1,\dots, u_l; t).
\end{align}
\item 
Coefficients of the formal distribution   (\ref{VpHL}) 
 \begin{align}\label{F1}
\mathcal{F}(u_1,\dots, u_l; t)=\sum_{\lambda\in \bZ^l} F_\lambda\, u_1^{\lambda_{1}} \dots u_{l}^{\lambda_{l}},\quad 
 \end{align}
have the form
\[
F_\lambda=  \Gamma_{-\lambda_1}^{+}  \Gamma_{-\lambda_2}^{+} \dots \Gamma_{-\lambda_l}^{+}(1) \, \in \Lambda[t].
\]
\item  For any $\lambda\in \bZ^l$, $F_{(m,\lambda)}=0$ for $m<<0$.
\item  $\{F_\lambda\}_{\lambda\in \bZ^l}$ defines  a family of symmetric polynomials  $F_\lambda= F_\lambda(x_1,x_2,\dots x_n; t)$, $n\in \bN$, with the stability property:
\[
F_\lambda (x_1,\dots, x_n;t)=F_\lambda (x_1,\dots, x_n, 0;t).
\]
   \item 
Let $\lambda \in \bN^l$  be an integer vector of length $l$  with positive coordinates. Let $n \ge l$ and set $\lambda_{l+1}=\dots=\lambda_n=0$. Then the corresponding coefficient  $F_\lambda$  in the expansion   (\ref{F1}) can be identified with the  symmetric polynomial   in variables $(x_1,x_2,\dots, x_n)$ with coefficients in $\bC[t]$
 \begin{align}\label{Fl}
F_\lambda =F_\lambda (x_1,\dots, x_n;t)=
\frac{(1-t)^n}{\prod_{i=1}^{n-l}(1-t^i)} \sum_{\sigma \in S_n}\sigma
\left(x_{1}^{\lambda_1}\dots x_{n}^{\lambda_n}\prod_{i=1}^{n}
\prod_{i < j}\frac{x_{i}-tx_{j}} {x_{i}-x_{j}}\right).
\end{align}
 When $\lambda$ is a partition, (\ref{Fl}) is called Hall-Littlewood  polynomial \footnote{Polynomials (\ref{Fl})  correspond to polynomials  $Q_\lambda (x,t)$ in  notations of \cite{Md} III.2  (2.11).}.
\item  The set $\{F_\lambda\}$ of Hall-Littlewood polynomials labeled by partitions  form a linear basis of $\Lambda[t]$. 
\item When  $\lambda$ is a partition,   the specializations of Hall-Littlewood polynomials provide important  families of symmetric functions:   $F_\lambda (x_1,\dots, x_n;0)= s_\lambda (x_1, x_2,\dots) $ is  Schur symmetric function (\ref{defshur}),   and  $F_\lambda (x_1,\dots, x_n;-1)= q_\lambda (x_1, x_2, \dots)$   is Schur $Q$-function (\ref {defq}).
 \end{enumerate}
\end{theorem}
\begin{proof} 
\begin{enumerate}[label=(\alph*)]
\item 
For the the original proof  see \cite{Jing2}. One can also  deduce  relations from  the definition (\ref{phps1}),  (\ref{phps2}) and Lemma \ref{propHE},  see \cite{NC} for this approach. 
\item  Follows from (\ref{phps1}),  (\ref{phps2}) and Lemma \ref{propHE} applied to $\Gamma^+(u_1)\dots \Gamma ^+(u_l) \,  (1)$.
\item  Immediately follows from (\ref{VpHL}).
\item  Follows from  the fact that $\Gamma ^+(u)$ is a  quantum field.
\item  Is proved directly  in \cite{Md}. For a shorter  proof observe that vertex operators (\ref{phps1}),  (\ref{phps2})  and their coefficients do  not depend on  $(x_1, x_2,\dots, x_n )$, hence the resulting symmetric polynomial  $F_\lambda$  also does not depend on them.

\item For a partition $\lambda$ this statement is proved in \cite{Jing2, Md}. The proof in \cite{Md}, III.2 carries without any changes to show that the statement is true for any 
 vector $\lambda \in \bN^l$.

\item  and (h) are  discussed in \cite{Md}, III.2.
\end{enumerate}
\end{proof}
\begin{remark}
 Note that  if $\lambda\in\bZ^l$ contains zero or  negative entries, the coefficient $F_\lambda$ of (\ref{F1}) is  still  an element of $\Lambda[t]$, 
but it  is not described by the formula (\ref{Fl}),  which in this case would involve negative powers of $x_i$'s.
For example, coefficients  $F_{(k)}=0$ for $k<0$,  and $F_{(-1,3)}= (t^3-t^2+t-1)F_2+(t^2-t) F_{(1)}^2$. 
\end{remark}

\subsection{Linear transformation of vertex operators of  Hall-Littlewood polynomials}\label{sub-trans}

 Let  $A=(A_{i,j})_{{i,j}\in \bZ}$  be a  complex-valued matrix with the property  that  for any  $i\in \bZ$, $A_{i,j}=0$ for $j<<0$. Following Section \ref{S-change}, we define 
 $ \tilde \Gamma^+_{i}=\sum_{j\in \bZ}A_{i,j} \Gamma^+_{j}$, where operators $\Gamma^+_j$ are coefficients of  quantum fields (\ref{phps1}) 
  that realize Hall-Littlewood  polynomials. 

\begin{theorem}\label{HLchange}
Let $\lambda \in \bZ^l$ and let 
 \begin{align}\label{Ftil}
 \tilde F_{\lambda}= \tilde\Gamma^+_{-\lambda_1}\dots \tilde\Gamma^+_{-\lambda_l}(1).
 \end{align}
\begin{enumerate}[label=(\alph*)]
\item  $ \tilde F_{\lambda}$ is  a  finite  element of $\Lambda[t]$ and 
 \begin{align}\label{defnchange}
 \tilde F_{\lambda}=\sum_{\alpha\in \bZ^l} A_{-\lambda_1,-\alpha_1}\dots A_{-\lambda_l,-
 \alpha_l}   F_{\alpha},
  \end{align}
where  $\{F_{\alpha}\in \Lambda[t]\}$ are coefficients in the expansion  (\ref{F1}). 

 \item
 \label{stab}
The family of symmetric polynomials   $\tilde F_{\lambda}$ ($\lambda\in \bZ^l$)    satisfy the stability property:
 \[\tilde F_{\lambda}(x_1,\dots, x_n, 0; t)= \tilde F_{\lambda}(x_1,\dots, x_n; t).\]
\item   If $A$ is invertible, then for any $\alpha \in \bZ^l$
   \begin{align*}
F_{\alpha}&=\sum_{\lambda\in \bZ^l} (A^{-1})_{-\alpha_1, -\lambda_1}\dots (A^{-1})_{-\alpha_l ,-\lambda_l} \tilde F_{\lambda}.
 \end{align*} 
 
 \item \label{gen} If $A$ is invertible, then we can write a re-expansion of 
the formal distribution (\ref{genHL})  with coefficients $\tilde F_\lambda$ :
\[
\mathcal{F}(u_1,\dots, u_l; t)
= \sum_{\lambda\in \bZ^l}\tilde F_\lambda g_{\lambda_1}(u_1)\dots g_{\lambda_l}(u_l),
\]
where  
$g_k(u)= \sum_{s\in \bZ} (A^{-1})_{-s,-k} u^{s}$ are complex-valued formal  distributions.

\item \label{polynom}

 Assume that $A_{i,j}=0$ for all $i< 0$, $j\ge 0$, and
 that $A_{0,j}=\delta_{0,j}$:
 
 \[
\begin{array}{c|rrrr|c|rrrc}
\downarrow i & j \rightarrow &-3 &-2   &-1        & 0 &1&2&3 &\dots \\
\hline
\dots&  \dots &\dots &\dots &\dots &\dots &\dots &\dots &\dots &\dots \\
-3&\dots &*& *  & *  &0 & 0 & 0& 0&\dots \\
-2& \dots &* &* & *  & 0 & 0 & 0 & 0 &\dots \\
-1& \dots &* &*&* & 0  & 0 & 0  & 0 &\dots \\
\hline
 0& \dots &0&0&0& 1&0 &0& 0 & \dots\\
 \hline
1&\dots &* &*&*  &*&*&*&* &\dots \\
2 &\dots &*  &*  &*  &*&*&*&*&\dots\\
3& \dots &* & *   &*  &*&*&*&*&\dots \\
4& \dots &* & *   &*  &*&*&\quad *&*&\dots \\
\dots&  \dots &\dots&\dots &\dots &\dots  &\dots &\dots &\dots&\dots 
\end{array}
\]
 
Let  $\lambda\in \bN^{l}$,   
let $n \ge l$. Set $\lambda_{l+1}=\dots=\lambda_n=0$. Then 
the element  $\tilde F_{\lambda}\in \Lambda[t]$  can be identified with  a symmetric  polynomial in variables $(x_1,\dots, x_n)$ with coefficients in $\bC[t]$ given by
 \begin{align}\label{newHL}
 \tilde F_\lambda (x_1,\dots, x_n;t)
 =\frac{(1-t)^n}{\prod_{i=1}^{n-l}(1-t^i)} 
 \sum_{\sigma \in S_n}\sigma
\left(f_{\lambda_1}(x_{1})\dots f_{\lambda_n}(x_{n})\prod_{i=1}^{n}
\prod_{i < j}\frac{x_{i}-tx_{j}} {x_{i}-x_{j}}\right),
\end{align}
 where   $f_k(x)= \sum_{j=1}^{-M(k)} A_{-k,-j} x^j$ ($k\in \bN$)
are complex-valued polynomials   with zero constant coefficient ($f_k(0)=0$), $f_0(x)=1$,  and $M(k)$ is defined as in (\ref{Ms}). 
 \end{enumerate}
\end{theorem}

\begin{proof}
Statements (a), (c) and (d)   immediately  follow from  Theorem \ref{matrixA}.

\ref{stab} Due to stability property of  symmetric  polynomials   $F_\alpha$  and expansion (\ref{defnchange}) that involves coefficients that do not depend on $(x_1,\dots, x_n)$,  polynomials $\tilde F_{\lambda}$,  also satisfy stability property.

\ref{polynom} Let  $\lambda\in \bN^l$.
With the imposed restriction on matrix $A$, all the terms  $F_\alpha$ in (\ref {defnchange})  have the form (\ref{Fl}), and
 \begin{align*}
& \tilde F_{\lambda}(x_1,x_2,\dots, x_n;t)=\sum_{\alpha \in \bN^l} A_{-\lambda_1,-\alpha_1}\dots A_{-\lambda_l-\alpha_l}   
 F_{\alpha}(x_1,x_2,\dots, x_n;t)\\
 &=
 \sum_{\alpha \in \bN^l} A_{-\lambda_1,-\alpha_1}\dots A_{-\lambda_l-\alpha_l} 
  \frac{(1-t)^n}{\prod_{i=1}^{n-l(\alpha)}(1-t^i)} 
   \sum_{\sigma \in S_n}\sigma
\left( x_{1}^{\alpha_1}  \dots  x_{l}^{\alpha_l}x_{l+1}^{0}\dots x_n^{0} \prod_{i=1}^{n}
\prod_{i < j}\frac{x_{i}-tx_{j}} {x_{i}-x_{j}}\right),
  \end{align*}
  Note that due to the restriction on the matrix $A$,   for all  non-trivial terms in the sum  $l(\alpha)=l$, and that we can interpret 
  $x_k^0=\sum_{\alpha_k\in\bZ} A_{0,-\alpha_k}x_k^{\alpha_k}$.
Then we can write 
 \begin{align*}
 & \tilde F_{\lambda}(x_1,x_2,\dots, x_n;t)=  \frac{(1-t)^n}{\prod_{i=1}^{n-l}(1-t^i)} 
\sum_{\sigma\in S_n}
\left(\prod_{i=1}^{n} \sum_{\alpha\in \bZ^n} A_{-\lambda_i, -\alpha_i}x_i^{\alpha_i}
 \prod_{i=1}^{n}
\prod_{i < j}\frac{x_{i}-tx_{j}} {x_{i}-x_{j}}\right).
    \end{align*}
    The general condition that   for any   $A_{i,j}=0$ for any $i\in\bZ$ $j>>0$ and the specific form of the matrix $A$ imply that $ \sum_{\alpha\in \bZ^n} A_{-\lambda_i, -\alpha_i}x_i^{\alpha_i}$
    are polynomials with zero constant term for $\lambda_i>0$, and just constant polynomial $1$ for $\lambda_i=0$. This  proves (\ref{newHL}).
\end{proof}

\begin{remark}
Representation theory of infinite-dimensional  algebraic structures is based on applications of symmetric functions  that do not depend on a number of variables, rather than on symmetric polynomials.
This can be seen in the formulation of the boson-fermion correspondence, actions of $GL_\infty$, $S_\infty$, centers of universal enveloping algebras, etc.
Hence, for such  applications  the stability property of generalizations of classical families of symmetric functions is essential, and we pay special attention to it through the text. 
\end{remark}
\begin{corollary}\label{cor1}
Let $\{f_k(x)\}_{k\in \bZ\ge 0}$ be a sequence of complex-valued polynomials  with the property that $f_0(x)=1$ and $f_k(0)=0$ for all $k\in \bN$. 
Then the family of symmetric polynomials in variables $(x_1,\dots, x_n)$ labeled by partitions $\lambda$
 \begin{align}\label{newHL2}
 \tilde{\tilde {F}}_\lambda (x_1,\dots, x_n;t)
 =\frac{(1-t)^n}{\prod_{i=1}^{n-l}(1-t^i)} 
 \sum_{\sigma \in S_n}\sigma
\left(f_{\lambda_1}(x_{1})\dots f_{\lambda_n}(x_{n})\prod_{i=1}^{n}
\prod_{i < j}\frac{x_{i}-tx_{j}} {x_{i}-x_{j}}\right),
\end{align}
satisfies stability property
\[
 \tilde{\tilde {F}}_\lambda (x_1,\dots, x_n;t)= \tilde{\tilde {F}}_\lambda (x_1,\dots, x_n,0;t).
\]
\end{corollary}
\begin{proof}
Due  to Theorem \ref{HLchange},  \ref{polynom}  such polynomials can be interpreted as a  result of a linear transformation of  vertex operator presentation 
of Hall -Littlewood polynomials with a matrix $A$ defined by the coefficients of the given sequence of polynomials.  Hence by Theorem \ref{HLchange},  \ref{stab}, they form a family of 
stable symmetric polynomials. 
\end{proof}
\begin{remark}\label{rem_k}

If $A$ is not of the form as in  Theorem \ref{HLchange},  \ref{polynom}, formula (\ref{newHL}) cannot be applied to compute the values of $\tilde F_\lambda(x_1,\dots, x_n)$, as illustrated by the next example. 

\begin{example} Let $a\ne0$ and let 
 \[
 A= Id+aE_{-2,0}=\small{
\begin{array}{rrrr|c|rrrc}
\dots &1& 0& 0 &0 & 0 & 0& 0&\dots \\
 \dots &0 &1 & 0 & a& 0 & 0 & 0 &\dots \\
 \dots &0 &0&1 & 0  & 0 & 0  & 0 &\dots \\
\hline
 \dots &0&0&0& 1&0 &0& 0 & \dots\\
 \hline
\dots &0 &0&0 &0&1&0&0 &\dots \\
\dots &0 &0 &0  &0&0&1&0&\dots\\
\dots &0 & 0   &0  &0&0&0&1&\dots \\
\end{array}}
\]
Then 
\[
f_2(x)=x^2+a,\quad \text{ and}  \quad f_k(x)= x^k \quad  \text{for} \quad k\ne 2.
\]
By definition, 
$\tilde F_{(2)}= F_{(2)}+aF_{0}=  F_{(2)}+a$. 
However, $\tilde F_2(x_1,\dots, x_n)$ cannot be  computed by     (\ref{newHL})  since,
using \cite{Md} III.1 (1.4), 
 \begin{align*}
 \tilde{\tilde {F}}_{(2)} (x_1,\dots, x_n;t)
 &=\frac{(1-t)^{n}}{\prod_{i=1}^{n-1}(1-t^i)} 
 \sum_{\sigma \in S_n}\sigma
\left((x_1^2+a)\prod_{i=1}^{n}
\prod_{i < j}\frac{x_{i}-tx_{j}} {x_{i}-x_{j}}\right)
= F_{(2)}+ a(1-t^n).
\end{align*}
This example  illustrates that if  the condition $f_k(0)=0$ is omitted in the Corollary \ref{cor1},  the resulting polynomial may depend on the  number of variables  $n$.
\end{example}
\end{remark}

\section{Specializations of linearly transformed Hall-Littlewood Polynomials}\label{Sec-5}
 
 \subsection{Specialization  $t=0$} \label{Sec-schur}
Linearly transformed Hall-Littlewood polynomials  specialized at $t=0$ correspond  to Schur-like symmetric functions.  In addition to the  properties implied by Theorem \ref{HLchange},  several nice properties  are specific to  this specialization.
Observe that 
$
i_{u_i, tu_j}{(u_i-u_j)}/{(u_i-tu_j)}|_{t=0}=1-\frac{u_j}{u_i}.
$
Using the  Vandermonde  determinant and  the Jacobi\,-Trudi identity (\ref{JT})  the  specialization  of  (\ref{genHL})  at  $t=0$ is reduced to the formal distribution
  \begin{align}\label{Su}
\mathcal{S}(u_1,.., u_l)&=
\Gamma^+(u_1)|_{t=0}\dots \Gamma^+(u_l)|_{t=0} (1)
=\prod_{1\le i<j\le l}\left(  \frac{u_i-u_j}{u_i} \right)\prod_{i=1}^{l} H(u_i)\\
&=\sum_{\alpha\in \bZ^{l}} \det[ h_{\alpha_i-i+j}] u^{\alpha}
=\sum_{\alpha\in \bZ^l}  s_\alpha u^\alpha.\notag
\end{align}
When $\alpha$ is a partition,  by the Jacobi\,-\,Trudi identity the coefficient  $s_\alpha$  coincides with a  Schur symmetric function, so $\mathcal{S}(u_1,.., u_l)$  can be viewed as the  generating function for  Schur symmetric functions.
It is also known  that 
\begin{align}\label{SH}
\mathcal S(u_1,.., u_l)= \det [ u_i^{-j+i} H(u_i)]. 
\end{align}

Let  $A=(A_{i,j})_{{i,j}\in \bZ}$  be a complex-valued matrix with the property  that  for any  $i\in \bZ$, $A_{i,j}=0$ for $j<<0$. 
Consider the specialization  of  (\ref{Ftil})  at $t=0$: 
  \begin{align}\label{rsom}
  \tilde s_{\lambda}=  \tilde\Gamma^+_{-\lambda_1}|_{t=0}\dots \tilde\Gamma^+_{-\lambda_l}|_{t=0}\,(1).
  \end{align}
\begin{corollary}\label{thm-s}
\begin{enumerate} [label=(\alph*)]

\item 
 For any integer vector  $\lambda \in \bZ^{l}$  
  \begin{align}\label{sal}
 \tilde s_{\lambda}=\sum_{\alpha\in \bZ^{l}} A_{-\lambda_1,-\alpha_1}\dots A_{-\lambda_l,- \alpha_l}   s_{\alpha},
 \end{align}
 and, if $A$ is invertible, 
\begin{align*}
s_{\alpha}=\sum_{\lambda\in  \bZ^{l}} (A^{-1})_{-\alpha_1, -\lambda_1}\dots (A^{-1})_{-\alpha_l, -\lambda_l} \tilde s_{\lambda}.
\end{align*}

\item \label {thm-b}
(Analogue of the Jacobi\,-\,Trudi formula) For any integer vector  $\lambda \in \bZ^{l}$  
  \begin{align} \label{rsom2}
  \tilde s_{\lambda}=\det \left[ \tilde h_{\lambda_i; i-j} \right] _{i, j=1,\dots l},
\end{align}
where 
$\tilde h_{k; m}=\sum_{r\in \bZ}A_{-k,-r}h_{r-m}$. If $A$ is invertible, then $h_{k; m}$ are coefficients of expansion
\begin{align*}
\sum_{k\in \bZ}\tilde h_{k;m}g_k(u)= u^{m}H(u),
\end{align*}
where $g_k(u)=\sum_{s\in\bZ}(A^{-1})_{-s,-k} u^s$ and $H(u)$ is defined by (\ref{HE}).

\item Symmetric polynomials 
 $\{\tilde s_{\lambda}\}_{\lambda\in \bZ^{ l}}$  satisfy  stability property
 \[\tilde s_{\lambda}(x_1,\dots, x_n, 0)= \tilde s_{\lambda}(x_1,\dots, x_n).\]

\item Let $A$ be invertible. 
 Then we have the equality of $\Lambda$-valued formal distributions
\begin{align}\label{rex}
\mathcal{S}(u_1,\dots, u_l)=\sum_{\lambda\in \bZ^l}\tilde s_\lambda \, g_{\lambda_1}(u_1)\dots g_{\lambda_l}(u_l). 
\end{align}

 \item \label{pol2}
  Let $A$ be of the form as in Theorem \ref{HLchange},  \ref{polynom}.  Let  $\lambda\in \bN^{l}$ and 
let  $n \ge l$. Set $\lambda_{l+1}=\dots=\lambda_n=0$. Then 
 $\tilde s_{\lambda}\in \Lambda$  can be identified with  a symmetric  polynomial in variables $(x_1,\dots, x_n)$ 
 \begin{align*}
 \tilde s_\lambda(x_1,\dots, x_n)=
\frac{\det[ f_{\lambda_j}(x_i)x_i^{n-j}]}{\det[x_i^{n-j}]}, 
\end{align*}
where   $f_k(x)= \sum_{j=1}^{-M(k)} A_{-k,-j} x^j$ ($k\in \bN$)
are complex-valued polynomials   with zero constant coefficient, $f_0(x)=1$,  and $M(k)$ is defined as in (\ref{Ms}).

\end{enumerate}
\end{corollary}
\begin{proof}
(a), (c), (d)  (e) follow from Theorem \ref{HLchange} by specialization at $t=0$.
Let us prove (b).  From (a) and  Jacobi\,-\,Trudi formula (\ref{JT}), 
  \begin{align*}
  \tilde s_{(\lambda_1,\dots,\lambda_l)}&=
  \sum_{\alpha\in \bZ^{l}} A_{-\lambda_1,-\alpha_1}\dots A_{-\lambda_l, -\alpha_l}  
  \det[ h_{\alpha_i-i+j}]\\
&= \sum _{\sigma \in S_l}\sum_{\alpha\in \bZ^{l}} (-1)^{\sigma}  A_{-\lambda_1,-\alpha_1}\dots A_{-\lambda_l,- \alpha_l}  
  h_{\alpha_1-1+\sigma(1)}\dots   h_{\alpha_l-l+\sigma(l)}
  \\
  &=
  \det \left[\sum_{r\in \bZ}A_{-\lambda_i,-r}h_{r-i+j}\right] _{1\le i, j\le l},
\end{align*}
which proves (\ref{rsom2}).
Direct calculation proves the second part of the statement.
\end{proof}


  \subsection{Specialization  $t=-1$} \label{Sec-Qsch} Linearly transfomed Hall-Littlewood polynomials at
  $t=-1$  correspond  to  generalizations of Schur $Q$-functions. Formula  (\ref{genHL})  at  $t=-1$ reduces 
to the generating function of  Schur  $Q$-functions
  \begin{align*}
\mathcal{Q}(u_1,\dots, u_l)&=\mathcal F(u_1,.., u_l; t)|_{t=-1}=
\Gamma^+(u_1)|_{t=-1}\dots \Gamma^+(u_l)|_{t=-1} (1)\notag \\
&=\prod_{1\le i<j\le l} i_{u_i, u_j}\left(  \frac{u_i-u_j}{u_i+u_j} \right)\prod_{i=1}^{l}  Q(u).
\end{align*}
Here $Q(u)=\sum _{k\in \bZ}q_ku^k$  is   is a formal distribution with coefficients in $\mathcal B_{odd}$   defined by (\ref{shurq}),  and  the series expansion of  rational functions    $\frac{u_i-u_j}{u_i+u_j}$  in the regions $|u_j|<|u_i|$ for  $1 \le i<j\le l$ is 
\begin{align*}
i_{u_i, u_j}\left(\frac{u_i-u_j}{u_i+u_j}\right)=1+2\sum_{s\ge 1}(-1)^s\left(\frac{u_j}{u_i}\right)^s \in \bC[[u_j/u_i]].
\end{align*}
Expand
  \begin{align}\label{eq-qu}
\mathcal Q(u_1,\dots, u_l)=\sum_{\alpha\in \bZ^{l}} Q_\alpha u^{\alpha}.
\end{align}
When $\alpha$ is a strict  partition,  coefficient  $Q_\alpha$ coincides with Schur $Q$-function (\ref{defq}), see \cite{Md} III.8.
In particular, consider the coefficients of the  formal distribution
\begin{align*}
\mathcal Q(u,v)= \left(1+2\sum_{s\ge 1}(-1)^s\frac{v^s}{u^s}\right)Q(u) Q(v)= \sum_{r,s\in \bZ}Q_{r,s}u^r v^s.
\end{align*}
For $r>s\ge 0$ the coefficient
$Q_{r,s}=q_{r,s}$ defined by (\ref{qrs}).

Since 
\begin{align*}
\mathcal Q(u_1,.., u_{ 2l-1})=- \mathcal Q(u_1,.., u_{2l-1}, 0),
\end{align*}
it is sufficient to consider the case of  even number of variables $(u_1, \dots, u_{2l})$.

Let  $M=(\mathcal M_{i,j})_{i,j=1,\dots, 2l}$ be a skew-symmetric matrix with entries
\begin{align*}
\mathcal M_{i,j}=
\begin{cases}
\mathcal Q(u_i,u_j),& i<j,\\ 
0, &i=j,\\
-\mathcal Q(u_j, u_i), &i>j.
\end{cases}
\end{align*}
It is known that 
\begin{align*}
\mathcal Q(u_1,.., u_{ 2l})= \text{Pf}\, [\mathcal M_{i,j}].
\end{align*}

\begin{remark} Note that  for $i>j$  both $\mathcal M_{i,j}= - \mathcal M_{j,i}$,   are elements of $\bC[[u_j/u_i]]$. 
\end{remark}

For any $\alpha \in \bZ^{2l}$
the coefficient $Q_\alpha$ of the expansion (\ref{eq-qu}) is 
\begin{align}\label{Gamb}
Q_\alpha =\text{Pf}\, [q_{\alpha_i,\alpha_ j} ]_{i, j=1,\dots, 2l},
\end{align}
where  $q_{ab}$ is  defined  by (\ref{qrs}).

Let $A$ be a complex-valued matrix with the property  that  for any  $i\in \bZ$, $A_{i,j}=0$ for $j<<0$.  Consider  specialization  of 
 (\ref{Ftil})  at $t=-1$: 
 \begin{align}\label{qtild}
   \tilde Q_{\lambda}=  \tilde\Gamma^+_{-\lambda_l}|_{t=-1}\dots \tilde\Gamma^+_{-\lambda_1}|_{t=-1}(1).
   \end{align}
\begin{corollary}\label{thm-sq}
\begin{enumerate} [label=(\alph*)]

\item 
 For any  $\lambda \in \bZ^{l}$,
 \begin{align}\label{salq}
 \tilde Q_{\lambda}=\sum_{\alpha\in \bZ^{l}} A_{-\lambda_1,-\alpha_1}\dots A_{-\lambda_l, -\alpha_l}   Q_{\alpha}.
 \end{align}
 If $A$ is invertible, then
\begin{align*}
Q_{\alpha}=\sum_{\lambda\in  \bZ^{l}} (A^{-1})_{-\alpha_1, -\lambda_1}\dots (A^{-1})_{-\alpha_l, -\lambda_l} \tilde Q_{\lambda}.
\end{align*}

\item \label {thm-b}
(Analogue of Pfaffian formula) For any integer vector  $\lambda \in \bZ_{2l}$  
  \begin{align}\label{pfaf}
  \tilde Q_{(\lambda_1,\dots,\lambda_{2l})}=\text{Pf}\,\left[\sum_{k,r\in\bZ}A_{-\lambda_i,-k} A_{-\lambda_j,-r} q_{k,r} \right]_{i, j=1,\dots, 2l}.
\end{align}
\item Symmetric polynomials 
 $\{\tilde Q_{\lambda}\}_{\lambda\in \bZ_{ l}}$  satisfy the  stability property:
 \[\tilde Q_{\lambda}(x_1,\dots, x_n, 0)= \tilde Q_{\lambda}(x_1,\dots, x_n).\]

\item Let $A$ be invertible. 
Consider $
g_k(u)= \sum_{s\in \bZ} (A^{-1})_{-s,-k} u^{-s}
$.
 Then  
\begin{align}\label{rexq}
 \mathcal Q(u_1,\dots, u_l)=
\sum_{\lambda\in \bZ^l}\tilde Q_\lambda \, g_{\lambda_1}(u_1)\dots g_{\lambda_l}(u_l).  
\end{align}

 \item   
 Let $A$ be  as in Theorem \ref{HLchange},  \ref{polynom}.  Let  $\lambda\in \bN^{l}$.  
For  $n \ge l$ and set $\lambda_{l+1}=\dots=\lambda_n=0$. Then 
the element  $\tilde Q_{\lambda}\in \Lambda$  can be identified with  a symmetric  polynomial in variables $(x_1,\dots, x_n)$
 \begin{align*}
 \tilde Q_\lambda(x_1,\dots, x_n)= 2^l \sum_{\sigma \in S_n }\sigma \left(f_{\lambda_1 }(x_{1})\dots f_{\lambda_n }(x_n)
\prod _{i=1}^{n}\prod_{i<j}\frac{x_{i}+x_{j}}{x_{i}- x_{j}}\right),
\end{align*}
where   $f_k(x)= \sum_{i=1}^{-M(k)} A_{-k,-j} x^j$ ($k\in \bN$), 
are complex-valued polynomials   with zero constant coefficient, $f_0(x)=1$,  and $M(k)$ is defined as in (\ref{Ms}).

\end{enumerate}
\end{corollary}
\begin{proof}
(a), (c), (d), (e) follow from Theorem \ref{HLchange} by evaluation  at $t=-1$.

Let $M^\lambda_{i,j}=\sum_{k,r}A_{-\lambda_i,-k} A_{-\lambda_j,-r} q_{k,r}$. Note that $M^\lambda_{i,j}=-M^\lambda_{j,i}$ and that $M^\lambda_{i,i}=0$.
From (a) and (\ref{Gamb}),  
\begin{align*}
&\tilde Q_{\lambda_1,\dots, \lambda_{2l}}=
\sum_{\alpha\in \bZ^{2l}} A_{-\lambda_1,-\alpha_1}\dots A_{-\lambda_{2l}, -\alpha_{2l}}   Q_{\alpha_1,\dots \alpha_{2l}}
=\sum_{\alpha\in \bZ^{2l}} A_{-\lambda_1,-\alpha_1}\dots A_{-\lambda_{2l}, -\alpha_{2l}}   \text{Pf}\, [q_{\alpha_i, \alpha_j} ]\\
&=\sum_{\alpha\in \bZ^{2l}} A_{-\lambda_1,-\alpha_1}\dots A_{-\lambda_{2l}, -\alpha_{2l}}  
\sum_{\sigma\in S_{2l}^\prime} sgn(\sigma)q_{\alpha_{\sigma(1)},\alpha_{\sigma(2)}} \dots q_{\alpha_{\sigma(2l-1)},\alpha_{\sigma(2l)}} \\
&=\sum_{\sigma\in S_{2l}^\prime} sgn(\sigma) \sum_{\alpha\in \bZ^{2l}} 
 A_{-\lambda_{\sigma(1)},-\alpha_{\sigma(1)}} A_{-\lambda_{\sigma(2)},-\alpha_{\sigma(2)}}q_{\alpha_{\sigma(1)},\alpha_{\sigma(2)}}
 \dots  A_{-\lambda_{\sigma(2l-1)},-\alpha_{\sigma(2l-1)}} A_{-\lambda_{\sigma(2l)},- \alpha_{\sigma(2l)}} q_{\alpha_{\sigma(2l-1)},\alpha_{\sigma(2l)}} \\
 &=\text{Pf}\, [ M^{\lambda}_{i,j} ]_{i, j=1,\dots, 2l}.
\end{align*}
We used that 
$ A_{-\lambda_1,-\alpha_1}\dots A_{-\lambda_{2l},- \alpha_{2l}}  
=
 A_{-\lambda_{\sigma(1)},-\alpha_{\sigma(1)}}\dots A_{-\lambda_{\sigma(2l)}, -\alpha_{\sigma(2l)}}  
$ 
 for any $\sigma\in S_{2l}$.
This computation proves (b).
\end{proof}


\section{Polynomial  tau-functions  of the  KP  and the BKP hierarchy}\label{Sec-6}

In \cite{Sato}  M.\,Sato introduced the KP  hierarchy of evolution equations. The ideas were  further developed  for this and  other  examples of soliton type hierarchies  by the Kyoto school  \cite{DJKM4,  DJKM3,  DJKM2, JM}, and later by many other authors. 
In  \cite{Sato}    solutions  of the KP hierarchy were expressed   through  tau-functions. Polynomial tau-functions  form  an infinite Grassmann manifold. Schur polynomials, which are  Schur symmetric functions expressed as polynomials of power sums, are  examples of   tau-functions of the KP hierarchy   \cite{Sato}.
 Similarly  in \cite{You1, You2} it is   proved that  Schur $Q$-functions expressed through polynomials of power sums are examples of polynomial tau-functions   of the BKP, DKP and MDKP hierarchies. 
  
 Recently  all polynomial tau-functions  of several soliton hierarchies were described. 
 In  \cite{KL-KP, KL-BKP, KL-sKP}  it is demonstrated that any polynomial tau-function of the KP-hierarchy is obtained from a Schur polynomial by  certain shifts of arguments in the Jacobi-Trudi formula, and  that  any polynomial tau-function of the BKP or the DKP hierarchy is  described by a shift of arguments  in the Pfaffian formula  for  Schur $Q$-polynomials. Another approach  in \cite{KLR}  proves   that  any polynomial   tau-function  of the KP, the BKP and  the $s$-component KP  hierarchy  can be interpreted as a zero-mode of an appropriate combinatorial  generating function. 

In this note we show that  specializations considered in Sections \ref{Sec-schur}, \ref{Sec-Qsch}  provide  one more  description of  all  polynomial tau-functions of the KP and the BKP hierarchies. This observation will easily follow from the results of  \cite{KLR} with the  advantage  that the new formulation  immediately   implies that a number of  Schur-like symmetric functions that can be found in the literature provide polynomial tau-functions of the KP  hierarchy (when these symmetric functions are expressed as polynomials in power sums). For example, we recover as a particular case our earlier result  \cite {NRQ} that multiparameter Schur  $Q$-functions are tau-functions of the BKP hierarchy. More examples can be found in the end of this note.

\subsection{Polynomial  tau-functions  of the bilinear  KP identity} \label{subKPbil}

 Let $\{\psi^{\pm}_{k}\}$ be  the charged free fermions those action on the boson Fock space $\B$ is defined in Section \ref{Sec_ferminons}.
Let 
$
 \Omega= \sum _{k\in\bZ+1/2}\psi^+ _k\otimes \psi^-_{-k}
$.
 The {\it  bilinear KP identity}     \cite{DJKM3, DJKM2, Kac-bomb}  is the equation of the form
\begin{align}\label{binKP}
\Omega\,  (\tau \otimes \tau)=0
\end{align}
on a function $\tau=z^m\tau(p_1,p_2,\dots)$ from the formal completion of the space $\B^{(m)}=z^m\Lambda$, where
the ring of symmetric functions  $\Lambda$ is identified with the ring $\bC[p_1,p_2,\dots]$ of polynomials in power sums (see Section \ref{Sec-sym1}).
Non-zero solutions of  (\ref{binKP}) are called {\it tau-functions of the KP hierarchy}.  Accordingly, we will say that  a non-zero
solution of (\ref{binKP}) is a polynomial  tau-function,   if it is  a polynomial function in the variables  $(p_1,p_2,\dots)$ times $z^m$ (hence, an element of $\B ^{(m)}$ rather than  a completion of the space $\B^{(m)}$).

In \cite{KLR} polynomial tau-functions of the KP hierarchy  are described as coefficients of formal distributions of the form  $A_1(u_1)\dots A_l(u_l) \mathcal S(u_1,\dots, u_l)$. Here we state that polynomial tau-functions of the KP hierarchy  are linear transformations of vertex operators presentation of Schur  symmetric functions, and in particular cases, are coefficients of re-expansions  of  the formal distribution $\mathcal S (u_1,\dots, u_l)$ in another basis of formal distributions.  

  Let  $A$  be a complex-valued matrix with the property  that  for any  $i\in \bZ$, $A_{i,j}=0$ for $j<<0$. 
For any integer vector  $\lambda \in \bZ^{l}$  let  $\tilde s_\lambda$ be defined by (\ref{rsom}). Recall that  $\tilde s_\lambda$ can be also computed  by
 (\ref{sal}) or by   (\ref{rsom2}).

 \begin{theorem} 
\label{KPs}

 \begin{enumerate}[label=\alph*)]
 \item     $\tilde s_\lambda$ is a polynomial tau-function of the KP hierarchy. Any polynomial tau-function of the KP hierarchy  is of  the form $\tilde s_\lambda$  for an appropriate choice of matrix $A$.
 
  \item  If $A$ is invertible and $
g_k(u)= \sum_{r\in \bZ} (A^{-1})_{-r,-k} u^{r}
$, then 
 any coefficient $\tilde s_\lambda$ of the re-expansion of the generating function  (\ref{Su})  of Schur symmetric functions
 \[
  \mathcal S(u_1,\dots, u_l)
 =\sum_{\lambda\in \bZ^l}(-1)^{l+1}\tilde s_\lambda \, g_{\lambda_1}(u_1)\dots g_{\lambda_l}(u_l)
 \]
 is a polynomial tau-function of the KP bilinear hierarchy.
\end{enumerate}
 \end{theorem}
\begin{proof}
 \begin{enumerate}[label=\alph*)]
\item 
In \cite{KLR} Theorem 3.1 provides  the following description of polynomial tau-functions  of the KP hierarchy. 
Consider  a  collection  of  complex-valued  Laurent series  $A_1(u), \dots, A_l(u)$. 
 Let
 $\alpha\in \bZ^l$,   and let $T_\alpha\in \Lambda$ be the coefficient\footnote{ $T_{(\alpha_1,\dots,\alpha_l)}$  here  corresponds to $T_{(\alpha_l-l+1,\dots,\alpha_1)}$,   and (\ref{AQ}) to  $T(u_l,\dots, u_1)/u_l^{l-1}\dots u_1^{0}$    in notations of  \cite{KLR} (3.12).  } of $u^\alpha$
  in the expansion 
\begin{align}\label{AQ}
\sum_{\alpha\in \bZ^l} T_\alpha u_l^{\alpha_1}\dots u_1^{\alpha_l}.
=\prod_{i=1}^{l}A_i(u_i)  \mathcal S(u_1,\dots, u_l).
\end{align}
Then   all coefficients $T_\alpha$ are polynomial  tau-functions  of the KP hierarchy, and for any polynomial tau-function of the KP hierarchy  there exists  a  collection  of  Laurent  polynomials $A_1(u), \dots, A_l(u)\in \bC[u, u^{-1}]$   such that
$\tau= T_{(0,\dots, 0)}$  in the Laurent series expansion of (\ref{AQ}).

Let $\tilde s_\lambda$ be defined by (\ref{rsom}). 
By  (\ref{sal}), $\tilde s_\lambda=\sum_{\alpha\in\bZ^l}A_{-\lambda_1,-\alpha_1}\dots A_{-\lambda_l,-\alpha_l}s_\alpha$, and the sum is finite. Then $\tilde s_\lambda$
is the coefficient of  $u_1^0\dots u_l^0$ in (\ref{AQ}) with the choice $A_{i}(u)=A_{-\lambda_i,k}u^k$, $i=1,\dots,l$ (note $A_i(u)$ are Laurent series by the property $A_{ij}=0$ for $j<<0$). Then by the results of  \cite{KLR},  $\tilde s_\lambda$  is a polynomial tau-function.

The other way, 
let $\tau$  be a polynomial KP tau-function. Then for an appropriate choice of  Laurent polynomials 
 $A_i(u)=\sum_{k\in\bZ} a_{i,k}u^k$, $i=1,\dots,l$, it is    a zero-mode of (\ref{AQ}) and 
   $\tau= \sum_{\alpha\in\bZ^l}a_{1,-\alpha_1}\dots a_{l,-\alpha_l}s_\alpha$.
   Then by  (\ref{sal}) $\tau=\tilde s_\lambda$ for the transformation matrix
   \begin{align}\label{trA}
A_{-i,j}
=\begin{cases}
a_{i, j},& i =1, \dots, l, j\in \bZ,
\\
\delta_{i,j},&\text{otherwise}.
\end{cases}
\end{align}
   Since $A_i(u)$ are Laurent polynomials, the matrix $A$ satisfies the condition $A_{ij}=0$ for $j<<0$.
   
 \item The  coefficients of the re-expansion  (\ref{rex}) are exactly  of the form (\ref{rsom}), hence they are polynomial tau-functions of the KP hierarchy. 
\end{enumerate}

\end{proof}
 \begin{remark} We use invertibility of matrix $A$ to define basis terms $g_k(u)$ in re-expansion of the generating function $\mathcal S(u_1, \dots, u_l)$. This restriction  on $A$ does not allow us to interpret  any polynomial  tau-function of the KP hierarchy as a coefficient of such  re-expansion. Yet, we will see in the end of the note   that  many   interesting symmetric functions correspond to invertible linear transformations, and  for this reason are polynomial tau-functions  that  can be interpreted  as  such coefficients. 
 \end{remark}
 \begin{remark}
The authors of  \cite{HO} study properties of  transformations acting on the quantum fields in a fermionic Fock space within the view of  applications  to   the description of the tau-functions of the KP and BKP hierarchy. The transformations in  \cite{HO}  depend on an invertible  upper-triangular matrix. The authors  relate with their construction   Schur-like symmetric polynomials,  where in the top alternant  monomials $x^k$ are substituted by an arbitrary sequence of monic polynomials \cite{HO} (3.5). In some cases   these Schur-like symmetric polynomials may depend on the number of variables $n$, in which case they  do not extend to  symmetric functions. 
 \end{remark}
 

\subsection{Polynomial  tau-functions  of the bilinear  BKP identity} \label{Sec-BKP}
Let $\{\varphi_k\}$ be  neutral fermions those action on the boson Fock space $\B_{odd}$ is defined by (\ref{defphi}). Let
$
\Omega=\sum_{n\in \bZ} \varphi_n\otimes (-1)^n \varphi_{-n}.
$
The {\it bilinear BKP identity} \cite{DJKM3, DJKM2} is the  equation  of the form 
\begin{align}\label{BKPid}
\Omega ( \tau\otimes \tau)= \tau\otimes\tau
\end{align}
on elements  $\tau=\tau(p_1,p_3, p_5\dots)$ from the completion of $\B_{odd}$.
Non-zero solutions of (\ref{BKPid}) are called the  {\it tau-functions of the BKP  hierarchy}. We will say that a   solution of (\ref{BKPid})  is  a polynomial 
tau-function if  it is a polynomial function  in the  variables  $(p_1,p_3, p_5\dots)$ (hence, it is an element of  $\B_{odd}$ rather than its completion).

All polynomial tau-functions of the BKP hierarchy are described in \cite{KL-BKP, KLR}. Similarly, to $t=0$  case, we state here that  these  polynomial tau-functions are results of    linear transformations  of  vertex operators of Schur $Q$-functions and, in invertible cases, are  coefficients of  re-expansions  of  $\mathcal Q(u_1,\dots, u_l)$. The proof is again based on the results of \cite{KLR}.

  Let  $A$  be a complex-valued matrix with the property  that  for any  $i\in \bZ$, $A_{i,j}=0$ for $j<<0$. For any $\lambda\in \bZ^{l}$ let $\tilde Q_\lambda$ be defined by (\ref{qtild}). Recall that $\tilde Q_\lambda$  can be also computed by (\ref{salq}) or by (\ref{pfaf}).

\begin{theorem} 

 \begin{enumerate} [label=\alph*)]
 \item   
  $ \tilde Q_{\lambda}$ is a polynomial tau-function of the BKP hierarchy. Any polynomial tau-function of the BKP hierarchy  can be written in the form $ \tilde Q_{\lambda}$  for an appropriate choice of matrix $A$.
 
  \item  If $A$ is invertible and $
g_k(u)= \sum_{r\in \bZ} (A^{-1})_{-r,-k} u^{r}
$, then 
 any coefficient $\tilde Q_\lambda$ of the re-expansion of generating function  (\ref{eq-qu}) of Schur $Q$-functions
\[
\mathcal{Q}(u_1,\dots, u_l)=
\sum_{\lambda\in \bZ^l}\tilde Q_\lambda \, g_{\lambda_1}(u_1)\dots g_{\lambda_l}(u_l)
 \]
 is a polynomial tau-function of the BKP bilinear hierarchy.
\end{enumerate}
 \end{theorem}
  \begin{enumerate} [label=\alph*)]
 \item  
The proof is based on  \cite{KLR} Theorem 4.1.
Let $A_1(u),\dots, A_l(u)\in \bC[u, u^{-1}]$ be a collection of   Laurent polynomials. 
For any  $\alpha\in \bZ^l$ let $T_\alpha$ be the coefficient in the expansion 
\begin{align}\label{AQBKP1}
\sum_{\alpha\in \bZ^l}T_\alpha u_1^{\alpha_1}\dots u_l^{\alpha_l}=  \prod_{i=1}^{l}A_i(u_i) \,\mathcal{Q}(u_1, \dots, u_l).
\end{align}
Then   all coefficients $T_\alpha$ are polynomial  tau-functions  of the BKP hierarchy, and for any polynomial tau-function of the BKP hierarchy  there exists  a  collection  of  Laurent  polynomials $A_1(u), \dots, A_l(u)\in \bC[u, u^{-1}]$   such that
$\tau= T_{(0,\dots, 0)}$  in the Laurent series expansion of (\ref{AQBKP1}).

By  (\ref{salq}),  $\tilde Q_\lambda$  can be interpreted as 
the  coefficient of  $u_1^0\dots u_l^0$ in (\ref{AQBKP1}) with the choice $A_{i}(u)=A_{-\lambda_i,k}u^k$, $i=1,\dots,l$. Hence $\tilde Q_\lambda$ is  a polynomial tau-function of the BKP hierarchy. 
Any  polynomial tau-function $\tau$ is    a zero-mode of (\ref{AQBKP1}) for an appropriate choice of  Laurent polynomials 
 $A_i(u)=\sum_{k\in\bZ} a_{i,k}u^k$, $i=1,\dots,l$. Hence  by  (\ref{salq})  it can be identified  with $\tau=\tilde Q_\lambda$ 
with  the same  transformation matrix (\ref{trA}) as in the KP case.     
 \item is  clear from (\ref{rexq}).
\end{enumerate}

 \section{Examples  of Linear Transformations}\label{Sec-Ex}
 In this section we relate general construction of linear transformations to examples  of symmetric functions considered by other authors. 
  \subsection{Linear transformation  by a Toeplitz matrix }
  Let $A$ be a Toeplitz matrix with    constant  complex values  $(a_k)_{k\in \bZ}$  along diagonals: 
 \[ A_{i,j}= a_{j-i}\quad \text{ for all $i,j \in \bZ$}.
 \] 
 The desired condition  $A_{i,j}=0$ for   $j<<0$  imposes the  restriction  on the sequence  $a_k=0$ for $k<<0$. 
 Then 
 \[
 \tilde \Gamma_i=\sum_{j\in \bZ}a_{j-i}\Gamma_j,\quad 
 \tilde F_\lambda=\sum_{\alpha\in \bZ} a_{\lambda_1-\alpha_1} \dots a_{\lambda_l-\alpha_l} F_\alpha
 \]
 and
 \[
 \sum_{\alpha\in\bZ^l} \tilde F_\alpha u_1^{\alpha_1}\dots u_l^{\alpha_l}= A(u_1)\dots A(u_l)  F(u_1,\dots, u_l), 
 \]
 with Laurent series  $A(u)= \sum_{k\in \bZ} a_k u^k$. 
Since in this case matrix $A$ is not of the form Theorem \ref {HLchange}  \ref{polynom}, we cannot  use   (\ref{newHL}) directly to compute  the corresponding transformation of  Hall-Littlewood polynomial   $\tilde F_\lambda$ as a symmetric polynomial in variables $(x_1,\dots, x_n)$.

\subsection{Change of basis $x^k\mapsto (x+\dots +x^k)$}
Let $A$ be the block matrix of the form (\ref{block})
with upper-triangular  blocks
\begin{align*}
(A^-)_{ij}=
\begin{cases}
1, & i\le j<0\\
0, &\text{otherwise}.
\end{cases}
\quad \quad \quad 
(A^+)_{i,j}=
\begin{cases}
1, &i=j>0,\\
-1,&j=i+1>1\\
0, &\text{otherwise}.
\end{cases}
\end{align*}

\[
\begin{array}{c|rrrr|c|rrrc}
\dots&\dots &-3 &-2   &-1        & 0 &1&2&3 &\dots \\
\hline
\dots&  \dots &\dots &\dots &\dots &\dots &\dots &\dots &\dots &\dots \\
-3&\dots &1& 1 & 1  & 0 & 0 & 0  & 0 &\dots \\
-2& \dots &0 &1 & 1  & 0  & 0  & 0  & 0  &\dots \\
-1& \dots &0 &0&1 & 0  & 0 & 0  & 0 &\dots \\
\hline
 0& \dots &0&0&0 & 1 &0 &0& 0 & \dots\\
 \hline
1&\dots &0 &0&0  &0&1&-1&0 &\dots \\
2 &\dots &0  &0   &0  &0&0&1&-1&\dots\\
3& \dots &0 & 0   &0  &0&0&0&1&\dots \\
4& \dots &0 & 0   &0  &0&0&\quad 0&0&\dots \\
\dots&  \dots &\dots&\dots &\dots &\dots  &\dots &\dots &\dots&\dots 
\end{array}
\]
 Then $A^\vee=A^{-1}$ 
  and by Lemma \ref{Lem-sym} for $k\in\bN$
 \begin{align*}
 f_k(x)=g_{-k}(x^{-1})= x+\dots +x^k=\frac{x(1-x^k)}{1-x},\quad
 f_{-k}(x)=g_k(x^{-1})=\frac{x-1}{x^{k+1}}, \quad  f_0(x)=g_0(x)=1.
 \end{align*}
 The matrix $A$ satisfies the properties of  Theorem \ref{HLchange},  \ref{polynom}.
 Hence for any partition $\lambda $ the  corresponding   transformation of Hall-Littlewood polynomials provides a  symmetric polynomial in variables  $(x_1,\dots, x_n)$
  \begin{align}\label{HLcyclott}
 \tilde F_\lambda (x_1,\dots, x_n;t)=
 \frac{(1-t)^n}{\prod_{i=1}^{n-l}(1-t^i)} 
 \sum_{\sigma \in S_n}\sigma
\left((1-x_{1}^{\lambda_1})\dots (1-x_{l}^{\lambda_n})\prod_{i=1}^{n}
\frac{x_i}{1-x_i}
\prod_{i < j}\frac{x_{i}-tx_{j}} {x_{i}-x_{j}}\right),
\end{align}
 with   specialization  at $t=0$
 \begin{align}\label{s1}
 \tilde s_\lambda (x_1,\dots, x_n)=\frac{
 \det [x_i^{n-j+1}(1-x_i^{\lambda_j})(1-x_i)^{-1} ]
 }{
 \det[x^{n-j}_i]
 }.
 \end{align}
For any partition $\lambda$ the Jacobi\,-\,Trudi identity (\ref{rsom2}) reads in this example as 
 \[\tilde s_\lambda=\det\left[\sum_{k=1}^{\lambda_i} h_{k-i+j}\right]_{i,j=1,\dots,l}= \det\left[\tilde s_{(\lambda_i+j-i)}- \tilde s_{(j-i)}\right]_{i,j=1,\dots,l},
 \]
where we used that the particular case of this formula  $\lambda=(m)$ gives $\tilde s_{(m)}=h_1+\dots +h_m$.
  Let $\{\psi^{\pm}_k\}_{k\in \bZ}$ be the charged free fermions with the action on the ring of symmetric functions defined in Section \ref{Sec_ferminons}. This action provides the vertex operator presentation of classical Schur functions. Then, according to our general construction,  operators
  \begin{align}\label{psichuiu}
  \tilde \psi^+_{k-1/2}=\sum_{i}A_{k,i}\psi^+_{i-1/2},\quad   \tilde \psi^-_{k-1/2}=\sum_{i}A_{k-1,i-1}\psi^-_{i-1/2}
  \end{align}
   provide the  vertex operator presentation of symmetric functions $\tilde s_\lambda$, and by  Propostion \ref{tildeferm}   also satisfy relations (\ref{ckl}) of charged free fermions. 
 
 By the results of Section \ref{Sec-6}, symmetric function (\ref{s1}), expressed as a polynomial  in power sums, is a  tau-function of the KP hierarchy, and the  specialization of (\ref{HLcyclott}) at  $t=-1$,  expressed  as a polynomial in odd power sums, is a tau-function of  the BKP hierarchy. 
 
 We would like to illustrate also that  linear transformations of quantum fields can be viewed as a  source of curious identities. 
 For example,  for $|x|<|y|$  Lemma \ref{L-delta} implies 
\[
 \sum_{k>0}\frac{(1-x^k)(y-1)}{y^{k+1}\,(1-x)}=\frac{1}{x-y}.
 \]
 
 More generally, let's apply Theorem \ref{HLchange}  taking $l=1$:
 \begin{align}\label{casel=1}
 \mathcal F(u)=\sum_{k\in\bZ}\tilde F_{(k)} g_k(u).
 \end{align}
Note that $\tilde F_{(k)}=0$ for $k<0$, $\tilde F_{(0)}=1$.
Using  \cite{Md} III (2.9), we get for $k\in \bN$,
\[
\tilde F_k= \sum_{s=1}^k F_k=(1-t)\sum_i (x_i+\dots +x_i^k)\prod_{j\ne i}\frac{x_i-tx_j}{x_i-x_j}=
(1-t)\sum_i \frac{x_i(1-x_i^{k})}{1-x_i}\prod_{j\ne i}\frac{x_i-tx_j}{x_i-x_j}.
\]

Recall that $ \mathcal F(u)=E(-tu)H(u)=\prod_i \frac{1-x_itu}{1-x_iu}$ and that  $g_k(1/u)=u^k-u^{k+1}$ for $k>0$.
Then (\ref{casel=1}) provides the identity of formal distributions
 \[
\prod_{i}\frac{1-x_itu}{1-x_iu}=1+(1-t)\sum_{r=1}^\infty\, \sum_{i}\frac{x_i(1-x_i^{r})}{(1-x_i)}\prod_{i\ne j}\frac{x-tx_j}{x_i-x_j} \, (u^r-u^{r+1}), 
\]
 in analogy with the identity  \cite{Md} III (2.10).
\subsection{Multiparatmeter symmetric functions}\label{sec_mutly}
Consider an   infinite sequence of complex numbers   $a= ( a_1, a_2,\dots)$. For  $n\in \bN$ define  the  multiparameter powers of variables 
\begin{align*}
(x|a)_n= (x-a_1)(x-a_2)\dots (x- a_{n}),
\end{align*}
and set $(x|a)_{0} =1$.
The following transitions are well-known (\cite{ORV} Lemma 2.5, \cite{Ivanov} (10.2), \cite{Mo} Theorem 2.1). They can be proved by direct computation. 
\begin{lemma}\label{transitions}
For $n\in \bN$
\begin{align*}
(x|a)_{n}&=\sum_{k=0}^{n}  (-1)^{n-k}e_{n-k}(a_1,\dots, a_{n}) \,x^k, \\
\frac{1}{(x|a)_{n}} &=\sum_{k=n}^{\infty}h_{k-n}(a_1,\dots, a_{n}) \,x^{-k}, \\
x^n&=\sum_{k=0}^{n}  h_{n-k}(a_1,\dots, a_{k+1}) \,(x|a)_k, \\
x^{-n} &=\sum_{k=n}^{\infty}  (-1)^{n-k}e_{k-n}(a_1,\dots, a_{k-1}) \frac{1}{(x|a)_{k}}.
\end{align*}
\end{lemma}
Let $A$ be the upper-triangular block matrix 
with the non-zero entries 
\begin{align*}
\begin{cases}
A_{-i,-j}&= (-1)^{i-j}e_{i-j}(a_1,\dots, a_{i-1}),\quad i,j\in \bN, \\
A_{i,j}&= h_{j-i}(a_1,\dots, a_i),\quad  i,j\in \bN.\\
A_{i,0}=A_{0,i}&=\delta_{i,0},\quad  i\in \bZ.
\end{cases}
\end{align*}
Using short notations $h_r[s]=h_r(a_1,\dots, a_s)$ and $e_r[s]=e_r(a_1,\dots, a_s)$, the  matrix $A$ has the form
\[
\begin{array}{r|rrrr|c|rrrc}
\dots&\dots &-3 &-2   &-1        & 0 &1&2&3 &\dots \\
\hline
-3&\dots &1&-e_1[2] & e_2[2] & 0 & 0 & 0  & 0 &\dots \\
-2& \dots &0 &1 &-e_1[1]  & 0  & 0  & 0  & 0  &\dots \\
-1& \dots &0 &0&1& 0  & 0 & 0  & 0 &\dots \\
\hline
 0& \dots &0&0&0 & 1 &0 &0& 0 & \dots\\
 \hline
1&\dots &0 &0&0  &0&1&h_1[1]&h_2[1] &\dots \\
2 &\dots &0  &0   &0  &0&0&1&h_2[2]&\dots\\
3& \dots &0 & 0   &0  &0&0&0&1&\dots \\
4& \dots &0 & 0   &0  &0&0&0&0&\dots \\
\end{array}
\]
\begin{lemma}
 $A^{-1}=A^{\vee}$
\end{lemma}
\begin{proof} 
Let  $n\in \bN$. By Lemma \ref{transitions} we can write the sequence of transitions 
\begin{align*}
\frac{1}{x^n}&= \sum_{k=1}^{\infty} (-1)^{n-k}e_{k-n}(a_1,\dots, a_{k-1})\frac{1} {(x|a)_{k}}\\
&=  \sum_{k,r=1}^{\infty} (-1)^{n-k}e_{k-n} (a_1,\dots, a_{k-1})  h_{r-k}(a_1,\dots, a_{k})\frac{1}{x^r}= \sum_{k,r=1}^{\infty}   A_{-k,-n} A_{r,k}\frac{1}{x^r}.
\end{align*}
Comparing the  coefficients of $1/x^{r}$  on both sides  we obtain that 
$
\delta_{n,r} =   \sum_{k=1}^{\infty}   A_{-k,-n} A_{r,k}
$.
 \end{proof}
By Lemma \ref{Lem-sym}  for $n\in\bN$,
\begin{align*} 
f_k(x)= g_{-k}(x^{-1})=x(x|a)_{k-1}, \quad  f_{-k}(x)=g_k(x^{-1})=1/(x|a)_{k},
\quad  f_0(x)=g_0(x)=1.
 \end{align*}
For $|x|<|y|$ Lemma \ref{L-delta} implies the identity 
  \cite{ OO1, ORV}
\[
\sum_{k>0}
\frac{\quad (x|a)_{k-1}}{(y|a)_k}= \frac{1}{y-x}.
\]
 The identity of type (\ref{casel=1}) looks in this example 
\[
\prod_{i}\frac{u-{x_it}}{u-{x_i}}=1+(1-t)\sum_{r=1}^\infty\sum_{i}\frac{x_i(x_i|a)_{r-1}}{(u|a)_{r}}\prod_{i\ne j}\frac{x-tx_j}{x_i-x_j}.
\]

By Theorem \ref{HLchange},  \ref{polynom},  for any partition $\lambda $ the  corresponding   transformation of Hall-Littlewood polynomials provides a  symmetric polynomial in variables  $(x_1,\dots, x_n)$
  \begin{align}\label{HLmulty}
 \tilde F_\lambda (x_1,\dots, x_n;t)=
 \frac{(1-t)^n}{\prod_{i=1}^{n-l}(1-t^i)} 
 \sum_{\sigma \in S_n}\sigma
\left((x_{1}|a)_{\lambda_1-1}\dots (x_{l}|a)_{\lambda_n-1}
\prod_{i=1}^{n}
{x_i}
\prod_{i < j}\frac{x_{i}-tx_{j}} {x_{i}-x_{j}}\right),
\end{align}
 with   specialization  at $t=0$
 \begin{align}\label{sm1}
 \tilde s_\lambda (x_1,\dots, x_n)=\frac{
 \det [x_i^{n-j+1}(x_i|a)_{\lambda_j-1} ]
 }{
 \det[x^{n-j}_i]
 }.
 \end{align}
The entries of the determinant of the 
 the Jacobi\,-\,Trudi identity (\ref{rsom2})  in this case are given by 
  \begin{align*}
 \tilde h_{i;m}=\sum_{j=1}^{i} (-1)^{i-j}e_{i-j}(a_1,\dots a_{i-1}) h_{j-m}, \quad i\in \bN.
 \end{align*}
 They can be interpreted as coefficients of re-expansion of $H(1/u)$: 
 \begin{align}\label{Hmu}
1+ \sum_{k\in \bN}\frac{\tilde h_{k;m}}{(u|a)_k}= u^{m} H(1/u).
 \end{align}

 By the results of Section \ref{Sec-6}, symmetric function (\ref{sm1}), expressed as a polynomial  in power sums, is a  tau-function of the KP hierarchy, and the  specialization of (\ref{HLmulty}) at  $t=-1$,  expressed  as a polynomial in odd power sums, is a tau-function of  the BKP hierarchy.

Different   multiparameter analogues of  symmetric functions find  their applications  in representation theory, algebraic geometry, combinatorics.  
Without   aspirations  to get  even close to a complete overview,  we  mention  here some of the examples  that are most closely related to our construction.

Polynomials  of type (\ref{HLmulty}) appear in 
\cite{NN1,NN3, NN2},  
where  the authors generalize  Hall-Littlewood polynomials by the substitution $x_i^k\to (x|a)_k$ and the change  of ordinary addition to some formal group law. The authors study combinatorics   and relations of these polynomials to the flag bundles in the complex  cobordism theory.

One can note that   the ratio of two alternants (\ref {sm1}) stands between the definition of  Schur symmetric functions $s_\lambda$ and  the  factorial Schur functions, where the later are  symmetric polynomials that depend  on  a doubly-infinite
sequence of  parameters $a = (a_i)_{i\in \bZ}$. They  are  defined as
 \begin{align*}
  s^{F}_\lambda (x_1,\dots, x_n)=\frac{
 \det [(x_i|a)_{\lambda_j+n-j} ]_{1\le i, j\le n}
 }{
 \det[x^{n-j}_i]_{1\le i, j\le n}
 }.
 \end{align*}
On the contrary to $s_\lambda$ and $\tilde s_\lambda$,   symmetric polynomials $s^{F}_\lambda (x_1,\dots, x_n)$  do not enjoy the stability property. At the same time, their reach combinatorics is  widely studied in the literature. For example,  $s^{F}_\lambda (x_1,\dots, x_n)$   are known to be a special case of the double Schubert polynomials, satisfy vanishing properties  and the  Jacobi\,-\,Trudi identity,   and many  other analogues of standard properties of Schur symmetric functions: 
 \cite{ BL1, BL2, CL, GH, GG,  L1, LS,  M2, Mo, MS,O1},  \cite{Md} I-3 Examples 20–21, and many others. Factorial Schur functions  with the sequence of parameters  $(0,1,2,3,\dots)$ are closely related to  shifted Schur functions, which are  a family of  functions symmetric in shifted variables.   Their nowadays well-developed  theory  started with  \cite{OO1}. In relation to our setup we mention that  the vertex operator presentation of shifted symmetric functions was constructed in \cite {JR-genA}.  

The mentioned above examples  are compared  in  \cite{ORV}  with  another multiparameter family of symmetric functions. 
The auhtors of  \cite{ORV}  
introduce Frobenius-Schur  super-symmetric functions that correspond to a sequence of parameters  $a_i=\frac{2i+1}{2}$, and 
extend the definition to  a general multiparatmeter analogue with a general sequence $a=(a_i)_{i\in \bZ}$. We denote here  these generalizations   as $s_{\mu; a}^{[ORV]}$. These symmetric functions find 
 their applications in the   asymptotic character theory of the symmetric group.  
By  formula (3.4) in \cite{ORV}, this  multiparameter analogue satisfies the Jacobi\,-\,Trudi  identity 
\begin{align}\label{JTorv}
s_{\mu; a}^{[ORV]}= \det\left[h^{[ORV]}_{\mu_i-i+j; \tau^{1-j}a}\right],
\end{align}
where $\mu$ is  a partition, and 
$h_{k;\tau^{-r}a}^{[ORV]}$ are coefficients of the formal distribution re-expansion
\begin{align}\label{horv}
1+\sum_{k>0} \frac{h^{[ORV]}_{k;\tau^{-r}a}} {(u-a_{1-r})\dots (u-a_{k-r})}= H(1/u).
\end{align}
One can see from (\ref{horv}) and (\ref{Hmu}) that $h^{[ORV]}_{k;\tau^{0}a}=\tilde h_{k,0}$,
but in general $s_{\mu; a}^{[ORV]}$  is not of the form $\tilde s_{\mu}$. One probably can deduce a 
vertex operator presentation   of $s_{\mu; a}^{[ORV]}$  in the spirit of (\ref{rsom}) from (\ref{JTorv}), but it would most likely have more complicated structure  involving a shift   $(\tau^ra)_i=a_{i+r}$  at every step of application of a vertex operator.

Thus, (\ref{sm1})  resembles  different multiparameter analogues of  Schur functions considered in the literature, but does not  coincide with any of them. 

At the same time, the specialization $t=-1$ of  (\ref{HLmulty})   does coincide  with  introduced earlier  by other authors  the multiparameter  Schur Q-functions for a sequence of parameters $(0, a_1, a_2, \dots)$.
These  interpolation analogues of the classical Schur Q-functions  were studied combinatorially in \cite{Ivanov, 11}. 
In \cite{NRQ} we proved that multiparameter  Schur Q-functions are tau-functions  of the  BKP hierarchy.
The initial goal of this study, achieved in Section \ref{Sec-6}, was to generalize this result.
The factorial  Schur Q-functions, which are multiparameter  Schur Q-functions corresponding to  $(a_i=i-1)$ proved
to be useful in study of a number of questions of representation theory and algebraic geometry: 
\cite{1, I, IMN, IN1,15, Serg}


\subsection{The uniform shift }
Let's consider the special case of  linear transformation  of Section \ref{sec_mutly}
  with all parameter values  $(a_i=1)_{i\in \bN}$. Then for $k\in \bN$
\[
f_k(x)= x(x-1)^{k-1}=\sum_{j>0} (-1)^{k-j}{k-1\choose j-1}x^j, 
 \quad  f_{-k} =\frac{1}{(x-1)^{k}}=\sum_{j=1}^{\infty}{{j-1}\choose{k-1}}\frac{1}{x^j}, \quad  f_{0}=0.
\]
The blocks of  matrix $A$  of Section \ref{sec_mutly} become  Pascal's matrices
\begin{align}\label{Ashift}
\begin{cases} 
A_{-i,-j}&= (-1)^{i-j}{i-1\choose j-1} ,\quad i,j\in \bN, \\
A_{i,j}&= { j-1 \choose i-1},\quad  i,j\in \bN.\\
A_{i,0}=A_{0,i}&=\delta_{i,0},\quad  i\in \bZ.
\end{cases}
\end{align}

 Then  the polynomial 
  \begin{align}\label{H+1}
 \tilde F_\lambda (x_1,\dots, x_n;t)=
 \frac{(1-t)^n}{\prod_{i=1}^{n-l}(1-t^i)} 
 \sum_{\sigma \in S_n}\sigma
\left(x_1(x_{1}-1)^{\lambda_1-1}\dots x_n(x_{n}-1)^{\lambda_n-1}
\prod_{i < j}\frac{x_{i}-tx_{j}} {x_{i}-x_{j}}\right)
\end{align}
coincides with 
 the inhomogeneous Hall-Littlewood polynomial that  appear  in \cite{Bor14}   as a  degeneration   of a  dual  Hall-Littlewood-like rational symmetric functions. That is  polynomial $\tilde G_\lambda$  with $k=0$ in the notations of Section 8.2 of \cite{Bor14}. These rational symmetric  functions are rational  deformations of Hall-Littlewood polynomials, defined as partition functions for path ensembles in a  square grid  with assigned vertex weights.
 Accordingly, the  specialization at  $t=0$
 \begin{align}\label{sa=1}
 \tilde s_\lambda (x_1,\dots, x_n)=\frac{
 \det [(x_i-1)^{\lambda_j-1}x_i^{n-j+1} ]
 }{
{\prod_{i<j} (x_i-x_j)}.
 }
 \end{align}
corresponds to  inhomogeneous Schur polynomials  $G_\lambda^{(q=0)}$ with $k=0$ in Section 8.4 of \cite{Bor14}.

Then all the general  statements about linear transformations can be applied to  the examples of this section:
\begin{enumerate} [label=\alph*)]
\item 
Fromula (\ref{Ftil}) of Theorem \ref{HLchange} with $\tilde \Gamma^+_i=\sum A_{ij}\Gamma^+_j$,  matrix $A$ given by (\ref{Ashift})  and $\Gamma^+_i$ by (\ref{phps1}), provides  a  vertex operator presentation of  (\ref{H+1}), and 
fromula (\ref{defnchange}) gives the expression of (\ref{H+1}) in terms of Hall-Littlewood polynomials. 
Theorem \ref{HLchange}  \ref{stab}  implies that  (\ref{H+1}) are stable polynomials. 
\item
Corollary \ref{thm-s}  b)  with matrix $A$  given by (\ref{Ashift})  gives the  Jacobi\,-\,Trudi formula of  (\ref{sa=1})
$
 \tilde s_\lambda =
 \det \left[\sum_{r}(-1)^{\lambda_i-j}{{\lambda_i-1}\choose{r-1}}h_{r-i+j} \right]
$. 
\item By Theorem \ref{KPs}, the  family  of symmetric functions  (\ref{sa=1}),  expressed as polynomials of power sums,  can be interpreted as tau-functions of the KP hierarchy (\ref{binKP}). 
\item Since in the considered example   $A^{-1}=A^{\vee}$,  by Lemma \ref{Lem-sym}, $g_k(x)=f_{-k}(x^{-1})$.   Following  Corollary \ref{thm-s}, re-expand the generating function of 
Schur symmetric functions in the basis of monomials in formal distributions $g_k(u_i)$:
\[
\mathcal S(u_1,\dots, u_l)= \sum_{\lambda\in \bZ^l} \tilde s_\lambda g_{\lambda_1(u)}\dots g_{\lambda_l(u)}.
\]
Then by Corollary \ref{thm-s}  d), for any partition $\lambda$ symmetric function (\ref{sa=1}) is the coefficient of the monomial
\[
g_{\lambda_1(u)}\dots g_{\lambda_l(u)}= \frac{u_1^{\lambda_1}\dots u_l^{\lambda_l}}{(1-u_1)^{\lambda_1}\dots (1-u_l)^{\lambda_l}}
\]
of this re-expansion.  As before, here we identify rational functions with the appropriate expansions:   $\frac{u^{k}} {(1-u)^{k}}=\sum_{j=1}^{\infty}{{j-1}\choose{k-1}}{u^{j}}$.
With the change of variables $v_{i}=\frac{u_i}{1-u_i}$ we can state that 
for any partition $\lambda$ symmetric function $\tilde s_\lambda$ given by (\ref{sa=1}) is the coefficient of the monomial  $v^\lambda$
in the expansion of 
\[
\mathcal S(v_1/1+v_1,\dots, v_l/1+v_l)= \sum_{\lambda\in \bZ^l} \tilde s_\lambda v_1^{\lambda_1}\dots v_l^{\lambda_l}.
\]

\end{enumerate}

 \subsection{Grothendieck  polynomials}

In  \cite{Bor14} Section 8.4  similarity of polynomials (\ref{sa=1}) to  Grothendieck polynomials is also  pointed out. Grothendieck polynomials were introduced in \cite{LS}  as polynomial  representatives of Schubert classes in the Grothendieck ring. These polynomials and their variations  were further  studied by many authors, such as \cite{ FK,IN3,    IS, 12, Iwao1, Iwao2, MS1, MS2, MSdur} etc.

In \cite{MS1} formula (5.1) for a partition $\lambda$ the Grothendieck polynomial is  defined as
\begin{align}\label{Gr}
{G}_\lambda (x_1,\dots, x_n;\beta)=
\frac
{\det[x_i^{\lambda_j+n-j} (1+\beta x_i)^{j-1}  ]}
{\prod_{i<j} (x_i-x_j)}.
\end{align}
This formula can be used to express  the Grothendieck polynomials through polynomials  of type (\ref{sa=1})
Indeed, by the change of variables $x_i= 1+\beta y_i$ in (\ref{sa=1}) and introducing $\mu=(\mu_1,\dots, \mu_n)\in\bZ^n$ with  $\mu_{n-k+1}=\lambda_{k}-k$, we can write
\begin{align*}
\frac{\tilde s_{\lambda} (1+y_1\beta, \dots, 1+y_n\beta)}{e_n (1+y_1\beta,\dots, 1+y_n\beta) }\,
&=\,
\frac{
 \det [(\beta y_i)^{\lambda_j-1}(1+\beta y_i)^{n-j+1} ]
 }{
\prod_{i<j} (\beta y_{i}-\beta y_{j})
 {\prod_{i=1}^{n} (1+y_i\beta) }}\,
\,
\\
&=\frac{\beta^{|\mu|+n(n-1)/2}
 \det [ y_i^{\mu_{n-j+1}+j-1}(1+\beta y_i)^{n-j+1} ]
 }{\beta^{n(n-1)/2}
\prod_{i<j} ( y_{i}-y_{j})
 {\prod_{i=1}^{n} (1+y_i\beta) }}\,
=\,
\pm \beta ^{|\mu|}\,  G_\mu(y_1,\dots, y_n; \beta).
\end{align*}
\begin{remark}
 Note that if $\mu=(\mu_1,\dots, \mu_n)$   is a partition, then
$\lambda \in \bN^n$, and  $|\mu|=\sum(\lambda_i -i)$.
\end{remark}

In \cite{Iwao1, Iwao2} free fermions presentation of stable  Grothendieck polynomials and their duals are constructed. The Jacobi-Trudi identities that are 
provided in these papers  allow us to put these constructions in the  format of this note. 
By Proposition 3.9  in \cite{Iwao1}, stable  Grothendieck polynomials satisfy the Jacobi-Trudi identity 
\begin{align}\label{JT-G}
G_\lambda(x_1,\dots, x_n, \beta)=\det\left[
\sum_{m=0}^\infty { {i-l}\choose m}\beta^m G_{\lambda_i-i+j+m}, 
\right]
\end{align}
where $G_m=G_m(x_1,\dots, x_n;\beta)$ are defined through generating function of  Proposition 3.8  in \cite{Iwao1}, which we can interpret in our notations as 
\[
G(u)=\sum_{m=0}^{\infty} G_mu^m= 
i_{u,\beta}\left(\frac{1}{1+\beta/u}\right) E(\beta) H(u).
\]
Here  $i_{u,\beta}\left(\frac{u}{u+\beta}\right)=\sum_{i=0}^{\infty}\frac{(-\beta)^n}{u^n}$ and  $E(\beta)=\prod_{l=1}^{\infty}{(1+\beta x_l)}=\sum_{i\ge 0}e_{i}(x_1,x_2,\dots) \beta^{i}$.

Note that  formula (\ref{JT-G})  does  not depend on variables $(x_1,\dots, x_n)$. 

Let 
\[
\mathcal G(u_1,\dots, u_l)=\prod_{i=1}^{l} E(\beta) \left(1+{\beta}/{u_i}\right)^{i-l-1} \cdot  \mathcal S(u_1,\dots, u_l),
\]
where $ \mathcal S(u_1,\dots, u_l)$ is the generating function for Schur symmetric functions  given by (\ref{Su}), 
and we use binomial series expansion 
$
 \left(1+\frac{\beta}{u}\right)^{a} =i_{u,\beta} \left(1+\frac{\beta}{u}\right)^{a} =\sum_{r=0}^{\infty}{{a}\choose r}\frac{\beta^r}{u^r}
$
for any $a\in \bC$.
\begin{proposition} 
\begin{enumerate} [label=\alph*)]
\item
In the  expansion  $\mathcal G(u_1,\dots, u_l)=\sum_{\lambda\in \bZ^l} G_\lambda u^{\lambda}$   the coefficients $G_\lambda$ that correspond to 
partitions $\lambda$ are stable  Grothendieck polynomials. 
\item Vertex operator presentation of   stable  Grothendieck polynomials can be written in the form
\[
\mathcal G(u_1,\dots, u_l)=B^+(u_1)\dots B^+(u_l)\, (1),
\]
where 
\[B^+(u)= H(u)E^\perp(-u)E(\beta)=\Gamma^+(u)|_{t=0} \, E(\beta)
\]
 in notations of Sections \ref{secGKP} and   \ref{secHL}.
\end{enumerate}
\end{proposition}
\begin{proof}
\begin{enumerate} [label=\alph*)]
\item  By (\ref{SH})
\begin{align*}
\mathcal G(u_1,\dots, u_l)&=
\prod_{i=1}^{l} E(\beta) \left(1+{\beta}/{u_i}\right)^{i-l-1} \,\cdot   \det [u_{i}^{-j+i}H(u_i)]\\
&=\prod_{i=1}^{l}  \left(1+{\beta}/{u_i}\right)^{i-l}  \det [u_{i}^{-j+i} \left(1+{\beta}/{u_i}\right)^{-1} E(\beta)H(u_i)]
=\prod_{i=1}^{l}  \left(1+{\beta}/{u_i}\right)^{i-l}  \det [u_{i}^{-j+i} G(u_i)]
\\
&=
\prod_{i=1}^{l} \sum_{m_i\ge 0} {{i-l}\choose {m_i}}\frac{\beta^{m_i}}{u_i^{m_i}}\sum_{\sigma\in S_l}\sum_{\alpha\in \bZ^{l} } (-1)^{\sigma}
G_{\alpha_1}u_1^{\alpha_1+1-\sigma(1)}\dots G_{\alpha_l}u_l^{\alpha_l+l-\sigma(l)}
\\
&=\sum_{\lambda\in \bZ^{l} }  \sum_{\sigma\in S_l}  \sum_{m_i\ge 0} (-1)^{\sigma}
 {{1-l}\choose {m_1}}{\beta^{m_1}}G_{\lambda_1-1+\sigma(1)+m_1}\dots  {{0} \choose {m_l}}{\beta^{m_l}}G_{\lambda_l-l+\sigma(l)+m_l}u_1^{\lambda_1}\dots u_l^{\lambda_l}
 \\
 &=
\sum_{\lambda\in \bZ^{l} }\det\left[
 \sum_{m_i\ge 0} 
 {{i-l}\choose {m_i}}{\beta^{m_i}}G_{\lambda_i-i+j+m_i}
\right]
 u_1^{\lambda_1}\dots u_l^{\lambda_l} 
 =\sum_{\lambda\in \bZ^{l} } G_\lambda u^\lambda.
\end{align*}

\item
Using  relations of Lemma \ref{propHE}
\[
\left(1+\frac{\beta}{u}\right)E^{\perp}(-u) E({\beta})=E(\beta)E^{\perp}(-u),
\]
move all terms $E(\beta)$ in  the product $B^+(u_1)\dots B^+(u_l)\, (1)$ to the left,  and use that \\
$H(u_1)E^\perp(-u_1)\dots H(u_l)E^\perp(-u_l)(1)=\mathcal S(u_1,\dots, u_l)$ by (\ref{Su}).

\end{enumerate}
\end{proof}

\smallskip

Similarly, the definition of  dual stable  Grothendieck polynomials $g_\lambda(z_1, \dots z_n)$ in Section 4.1 of \cite{Iwao1}  is  followed by the Jacobi-Trudi formula 
\begin{align*}
g_\lambda(x_1,\dots x_n; \beta)=\det\left[
\sum_{m=0}^\infty { {1-i}\choose m}\beta^m h_{\lambda_i-i+j-m},
\right]
\end{align*}
(Proposition 4.4 in  \cite{Iwao1}, see also \cite{AY}).
Let 
\[
\mathcal J (u_1,\dots, u_l)=\prod_{i=1}^{l}  \left(1+{\beta}{u_i}\right)^{1-i}  \mathcal S(u_1,\dots, u_l),
\]
where $ \mathcal S(u_1,\dots, u_l)$ is the generating function for Schur symmetric functions   given by (\ref{Su}) and we use the binomial expansion
$
 \left(1+{\beta}{u}\right)^{a} =i_{u^{-1},\beta} \left(1+{\beta}{u}\right)^{a} =\sum_{r=0}^{\infty}{{a}\choose r}{\beta^r}{u^r}
$ for any $a\in \bC$.
 \begin{proposition} 
\begin{enumerate}  [label=\alph*)]
\item
In the  expansion  $\mathcal J(u_1,\dots, u_l)=\sum_{\lambda\in \bZ^l} g_\lambda u^{\lambda}$   the coefficients $g_\lambda$ that correspond to 
partitions $\lambda$ are dual stable Grothendieck polynomials. 
\item 
$
\mathcal J (u_1,\dots, u_l)=  J^+(u_1)\dots J^+(u_l)\, (1),
$
where 
\[J^+(u)= H(u)E^\perp(-u)H^\perp(-1/\beta)=\Gamma^+(u)|_{t=0} \, H^\perp(-1/\beta)
\]
 in notations of Sections \ref{secGKP} and   \ref{secHL}.
 \item 
 Dual stable  Grothendieck polynomials $g_\lambda(x_1, \dots x_n)$, expressed as polynomials in power sums $p_1,p_2,\dots$ are polynomial  tau-functions of the KP hiearchy
 (\ref{binKP}). 
\end{enumerate}
\end{proposition}
 \begin{proof} Similar calculations show that 
 \begin{enumerate}  [label=\alph*)]
 \item
 \begin{align*}
\mathcal J(u_1,\dots, u_l)&=
\prod_{i=1}^{l} \left(1+{\beta}{u_i}\right)^{1-i} \,\cdot   \det [u_{i}^{-j+i}H(u_i)]\\
&=
\prod_{i=1}^{l} \sum_{m_i\ge 0} {{1-i}\choose {m_i}}\beta^{m_i}{u_i^{m_i}}\sum_{\sigma\in S_l}\sum_{\alpha\in \bZ^{l} } (-1)^{\sigma}
h_{\alpha_1}u_1^{\alpha_1+1-\sigma(1)}\dots h_{\alpha_l}u_l^{\alpha_l+l-\sigma(l)}
\\
&=\sum_{\lambda\in \bZ^{l} }  \sum_{\sigma\in S_l}  \sum_{m_i\ge 0} (-1)^{\sigma}
 {{0}\choose {m_1}}{\beta^{m_1}}h_{\lambda_1-1+\sigma(1)-m_1}\dots  {{1-l} \choose {m_l}}{\beta^{m_l}}h_{\lambda_l-l+\sigma(l)-m_l}u_1^{\lambda_1}\dots u_l^{\lambda_l}
 \\
 &=
\sum_{\lambda\in \bZ^{l} }\det\left[
 \sum_{m_i\ge 0} 
 {{1-i}\choose {m_i}}{\beta^{m_i}}h_{\lambda_i-i+j-m_i}
\right]
 u_1^{\lambda_1}\dots u_l^{\lambda_l} 
 =\sum_{\lambda\in \bZ^{l} } g_\lambda u^\lambda.
\end{align*}

\item   
Using  relations of Lemma \ref{propHE}
\[
\left(1+{\beta}{u}\right)H^{\perp}(-1/\beta) H(u)=H(u)H^{\perp}(-1/\beta),
\]
move all terms $H^\perp\left(-1/\beta \right)$ in  the product $J^+(u_1)\dots J^+(u_l)\, (1)$ to the  right,  and use that $H^\perp\left({-1}/{\beta}\right)(1)=1$ and that 
$H(u_1)E^\perp(-u_1)\dots H(u_l)E^\perp(-u_l)(1)=\mathcal S(u_1,\dots, u_l)$ by (\ref{Su}).

\item

By part a),  any dual stable Grothendieck polynomial  $g_\lambda$ is a coefficient of the series $\mathcal J(u_1,\dots, u_l)$. Note that  it is obtained  from  $ \mathcal S(u_1,\dots, u_l)$ 
by multiplication by power series $A_i(u_i)=\sum_{r=0}^{\infty}{{1-i}\choose r}{\beta^r}{u_i^r}$.
Then by  \cite{KLR} Theorem 3.1 symmetric function  $g_\lambda$ is a polynomial  tau-functions  of the KP hierarchy (\ref{binKP}).
\end{enumerate}

 \end{proof}


\begin{thebibliography}{99}




\bibitem{1}  A.~Alldridge, S.~Sahi, H.~Salmasian, 
{\em Schur Q-functions and the Capelli eigenvalue problem for the Lie superalgebra q(n)},
  Representation Theory and Harmonic Analysis on Symmetric Spaces, Contemp. Math.,
Vol. 714, Amer. Math. Soc., Providence, RI, 2018, 1-21.

\bibitem{AY}
A.~ Amanov, D.~Yeliussizov
{\em Determinantal formulas for dual Grothendieck polynomials}, 	arXiv:2003.03907 


\bibitem{BL1} 
L.~Biedenharn, J.~Louck, 
{\em  A new class of symmetric polynomials defined in terms of tableaux},
 Advances in Appl. Math. 10 (1989), 396–438.
 
\bibitem{BL2} 
L.~Biedenharn, J.~Louck, 
{\em Inhomogeneous basis set of symmetric polynomials defined by tableaux}, Proc. Nat. Acad. Sci. U.S.A. 87 (1990), 1441–1445.

\bibitem{Bor14}  
 A.~Borodin, 
{\em On a family of symmetric rational functions},
Adv. Math. 306 (2017), 973–1018.

 
 \bibitem{CL} W. Y. C.~Chen, J. D.~Louck, 
{\em The factorial Schur function},
 J. Math. Phys. 34 (1993), 4144–4160.
 
 
 
 
 \bibitem{DJKM4} 
 E.~Date, M.~Jimbo, M.~Kashiwara, T.~Miwa, \emph{Operator Approach to the Kadomtsev-Petviashvili Equation. Transformation Groups for Soliton Equations III}, J. Phys. Soc. Jpn. {\bf 50}, (1981), no. 11, 3806--3812.
   
    \bibitem{DJKM1} 
  E.~Date, M.~Jimbo, M.~Kashiwara, T.~Miwa,  \emph{ Transformation groups for soliton equations. Euclidean Lie algebras and reduction of the KP hierarchy}, Publ. Res. Inst. Math. Sci. {\bf 18} (1982), 1077--1110. 
    
      \bibitem{DJKM3} 
    E.~Date, M.~Jimbo, M.~Kashiwara, T.~Miwa,  \emph{Transformation groups for soliton equations IV. A new hierarchy of soliton equations of KP type}, Physica 4D (1982), 343--365. 
  
  \bibitem{DJKM2} 
    E.~Date, M.~Jimbo, M.~Kashiwara, T.~Miwa, \emph{Transformation groups for soliton equations}, 
  in: Nonlinear integrable systems -- classical theory and quantum theory eds M. Jimbo and T. Miwa, World Scientific, (1983), 39--120.

 \bibitem{FK}
S.~Fomin,  A.~Kirillov, 
{\em Grothendieck polynomials and the Yang-Baxter equation}, Formal power series and algebraic combinatorics/ Papers from the 6th Conference (FPSAC '94) held in New Brunswick, NJ, May 23–27 (1994). Center for Discrete Mathematics and Theoretical Computer Science (DIMACS),183--189. 
 
 
  
  \bibitem{GH} I.~Goulden,  A.~Hamel, 
{\em Shift operators and factorial symmetric functions}, J. Comb. Theor. A. 69 (1995), 51–60.

  \bibitem{GG} 
 I.~Goulden, C.~Greene,
{\em  A new tableau representation for supersymmetric Schur functions}, 
J. Algebra 170 (1994), 687–704.


\bibitem{HO}
J.~Harnad, A.~Yu.~ Orlov, 
{\em Polynomial KP and BKP $\tau$-functions and correlators},
Ann. Henri Poincaré 22 (2021), no. 9, 3025–3049.

  
\bibitem{IN3}
T.~Ikeda, H.~Naruse,
{\em K-theoretic analogues of factorial Schur P- and Q-functions},
Adv. Math. 243 (2013), 22–66.

 
\bibitem{IS}
T.~Ikeda,  T.~Shimazaki,
{\em A proof of K-theoretic Littlewood-Richardson rules by Bender-Knuth-type involutions},
Math. Res. Lett. 21 (2014), no. 2, 333–339.

\bibitem{I}
T.~ Ikeda,
 {\em Schubert classes in the equivariant cohomology of the Lagrangian Grassmannian},
  Adv. Math. 215 (2007), 1- 23. 
  
 \bibitem{IMN} 
T.~Ikeda, L.C.~ Mihalcea, H.~Naruse,
{\em Factorial P- and Q-Schur functions represent equivariant quantum Schubert classes},
 Osaka J. Math. 53 (2016), 591619.
 
\bibitem{IN1}
 T.~Ikeda, H.~Naruse,
{\em  Excited Young diagrams and equivariant Schubert calculus}, 
Trans. Amer. Math. Soc. 361 (2009), 5193- 5221.


 
 \bibitem{Ivanov} 
 V.N.~Ivanov,{\em Interpolation analogues of Schur Q-functions}, J. Math. Sci. 131 (2005), 5495–5507.

\bibitem{Iwao1}
S.~Iwao,  
{\em Grothendieck polynomials and the boson-fermion correspondence},
 Algebr. Comb. 3 (2020), no. 5, 1023 -- 1040.
 
 \bibitem{Iwao2}
 S.~Iwao,  
 {\em Free-fermions and skew stable Grothendieck polynomials}, 	arXiv:2004.09499. 


  
  \bibitem{JM}
  M.~Jimbo,  T.~Miwa,  \emph{Solitons and infinite-dimensional Lie algebras}, 
Publ. Res. Inst. Math. Sci. {\bf 19} (1983), no. 3, 943 --1001. 
  

  

 \bibitem{Jing3} 
  N.~Jing, {\em Vertex operators, symmetric functions, and the spin group $\Gamma_n$},
J. Algebra {\bf 138} (1991), no. { 2}, 340--398. 

  \bibitem{Jing2} 
  N.~Jing,  {\em Vertex operators and Hall-Littlewood symmetric functions},
 Adv. Math. {\bf 87}(1991), no.{2}, 226--248.  

    \bibitem {JR-genA} N.~Jing, N.~Rozhkovskaya, 
  {\em Generating functions for symmetric and shifted symmetric functions},  Journal of Combinatorics Volume 10, Number 1, 111--127, 2019.



 \bibitem{Kac-begin}
 V.~G.~Kac,{\em Vertex algebras for beginners.}
2nd ed., University Lecture Series, {\bf10}. Amer. Math. Soc., Providence, RI, 1998. 


\bibitem{KL-KP} 
 V.~G.~Kac, J. W.~van de Leur, {\em Equivalence of formulations of the MKP hierarchy and its polynomial tau-functions},
 Jpn. J. Math. {\bf13} (2018), no. 2, 235--271. 
  
\bibitem{KL-BKP}
 V.~G.~Kac, J. W.~van de Leur, {\em Polynomial tau-functions of BKP and DKP hierarchies},
J. Math. Phys. {\bf 60} (2019), no. 7. 

\bibitem{KL-sKP} 
V.~G.~Kac, J. W.~van de Leur, {\em Polynomial tau-functions for the  multi-component KP hierarchy}, 
arXiv:1901.07763.
 
 \bibitem{KLR}
 V.~G.~Kac, J. W.~van de Leur, N.~Rozhkovskaya
{\em Polynomial tau-functions of the KP, BKP, and the s-component KP hierarchies.}
J. Math. Phys. 62 (2021), no. 2, Paper No. 021702.

      \bibitem{Kac-bomb}

 V.~G. Kac, A.~K. Raina,  {\em Bombay lectures on highest weight representations of infinite-dimensional Lie algebras}, 1st ed.,  Advanced Series in Mathematical Physics, 2. World Scientific Publishing Co., NJ(1987).
  \\
V.~G. Kac, A.~K. Raina, N.~Rozhkovskaya, {\em Bombay lectures on highest weight representations of infinite dimensional Lie algebras}, 2nd ed., World Scientific Publishing Co., NJ (2013). 

\bibitem{11}  S.~Korotkikh,
{\em  Dual multiparameter Schur Q-functions},
 J. Math. Sci. 224 (2017), 263- 268.



\bibitem{L1} A. Lascoux, 
{\em Puissances ext\'erieures, d\'eterminants et cycles de Schubert}, Bull Soc. Math.
France 102 (1974), 161–179. 


\bibitem{LS}
 A.~ Lascoux,  M.-P.~Schutzenberger,
 {\em Symmetry and flag manifolds. Invariant Theory.} 
  Lecture Notes in Mathematics 996, 1983, pp. 118–144.
  


   
   \bibitem{12} 
C.~ Lenart,
{\em Combinatorial Aspects of the K-Theory of Grassmannians.}
 Ann. Comb. 4 (2000), no. 1, 67–82.

   
     \bibitem{Md}
  I.~G.~Macdonald, {\em Symmetric functions and Hall polynomials},
 2nd ed., Oxford Univ. Press, New York, 1995. 
 
 \bibitem{M2}  I. G. Macdonald
 {\em Schur functions: theme and variations}, 
 Publ. I.R.M.A. Strasbourg , 498/S–27, Actes 28-e Seminaire Lotharingien (1992), 5–39.

\bibitem{Mo} 
A.~ Molev, 
{\em Factorial supersymmetric Schur functions and super Capelli identities},
 Kirillov’s Seminar on Representation Theory (G. Olshanski, ed.), American Mathematical Society Translations (2), Vol. 181, Amer. Math. Soc., Providence, R.I., 1997, pp. 109–137.

\bibitem{MS} A.~Molev, B.~Sagan, 
{\em A Littlewood–Richardson rule for factorial Schur functions},
Trans. Amer. Math. Soc. 351 (1999), 4429–4443.



\bibitem{MS1}
 K. Motegi,  K. Sakai. 
{\em 
Vertex models, TASEP and Grothendieck polynomials},
 J. Phys. A: Math. Theor.46 (2013), 355201.

\bibitem{MS2} 
K.~ Motegi,  K. Sakai. 
{\em K-theoretic boson-fermion correspondence and melting crystals},
J. Phys. A 47 (2014), no. 44, 445202.
 
 \bibitem{MSdur}
K.~ Motegi, T.~Scrimshaw,
{\em Refined dual Grothendieck polynomials, integrability, and the Schur measure}, 
S\'em. Lothar. Combin. 85B (2021), Art. 23.
 
 
 \bibitem{NN1} M.~Nakagawa, H.~Naruse,
 {\em Generalized (co)homology of the loop spaces of classical groups and the universal factorial Schur P- and Q-functions}, 
 Schubert calculus—Osaka 2012, 337–417, Adv. Stud. Pure Math., 71, Math. Soc. Japan, [Tokyo], 2016.
 

 \bibitem{NN3} M.~Nakagawa, H.~Naruse,
{\em Universal Gysin formulas for the universal Hall-Littlewood functions},  An alpine bouquet of algebraic topology, 201–244,
Contemp. Math., 708, Amer. Math. Soc., [Providence], RI, 2018.

 
  \bibitem{NN2}H.~Naruse,
 {\em Elementary proof and application of the generating functions for generalized Hall-Littlewood functions}, 
J. Algebra 516 (2018), 197–209.


\bibitem{15}  M.~Nazarov, 
{\em  Capelli identities for Lie superalgebras}, 
 Ann.Sci. Ecole Norm. Sup. (4) 30 (1997), 847-872.


\bibitem{NC} G.~ Necoechea, N.~ Rozhkovskaya,
{\em Generalized vertex operators of Hall-Littlewood polynomials as twists of charged free fermions},
J. Math. Sci. (N.Y.) 247 (2020), no. 6, Problems in mathematical analysis. No. 102, 926–938. 






\bibitem{O1} A.~Okounkov, 
{\em Quantum immanants and higher Capelli identities},
 Transformation Groups 1 (1996), 99–126. 

\bibitem{OO1} A.~Okounkov, G.~Olshanski, 
{\em Shifted Schur functions,}  St. Petersburg Math. J. 9 (1998), 239–300.

 
 \bibitem{ORV} 
 G.~Olshanski, A.~Regev, A.~ Vershik,
{\em Frobenius-Schur functions }
Progr. Math., 210, Studies in memory of Issai Schur (Chevaleret/Rehovot, 2000), 251–299.


   \bibitem{NRQ} 
  N.~Rozhkovskaya, \emph{Multiparameter Schur $Q$-functions are solutions of the BKP hierarchy},  SIGMA Symmetry Integrability Geom. Methods Appl.  {\bf 15} (2019), 065.  
  
  

  
  
     \bibitem{Serg} 
     S.~Sahi, H.~Salmasian, V.~Serganova
{\em The Capelli eigenvalue problem for Lie superalgebras}
Math. Z. 294 (2020), no. 1-2, 359–395.  


  
  \bibitem{Sato}
  M.~Sato, {\em Soliton equations as dynamical systems  on infinite-dimensional Grassmann manifold}, RIMS Kokyuroku, {\bf 439} (1981), 30--46.
  
 
\bibitem{Stan}
  R.~P. Stanley, 
{\em Enumerative combinatorics}. {Vol. 2}, Cambridge Univ. Press, Cambridge, (1999). 



  \bibitem{You1} 
   Y.~You, {\em Polynomial solutions of the BKP hierarchy and projective representations of symmetric groups}, in Infinite-dimensional Lie algebras and groups (Luminy-Marseille, 1988) Adv. Ser. Math. Phys. {\bf 7} (1989) 449--464.
  
  \bibitem{You2} 
  Y.~You, {\em DKP and MDKP hierarchy of soliton equations}, Physica D 50 (1991), 429--462.
 

 
  \bibitem{Zel} A.~  Zelevinsky,
{\em Representations of finite classical groups. A Hopf algebra approach.}  Lecture Notes in Mathematics, 869. Springer-Verlag, Berlin-New York, 1981.
 

 
  \end{thebibliography}
\end{document}